\documentclass[12pt,reqno]{amsart}

\usepackage{amsfonts}
\usepackage{version}
\usepackage{amssymb}
\usepackage[all]{xy}
\newcommand{\de}{\partial}
\newcommand{\R}{\mathbb R}

\newcommand{\al}{\alpha}

\newcommand{\Z}{\mathbb Z}
\newcommand{\C}{\mathbb C}

\newcommand{\B}{\mathbb B}

\newcommand{\Hol}{{\sf Hol}}
\newcommand{\Aut}{{\sf Aut}}
\newcommand{\D}{\mathbb D}

\newcommand{\la}{\langle}
\newcommand{\ra}{\rangle}

\newcommand{\N}{\mathbb N}

\newcommand{\Ha}{\mathbb H}

\def\v{\varphi}

\def\Re{{\sf Re}\,}
\def\Im{{\sf Im}\,}

\def\H{{\mathbb H}}

\newtheorem{theorem}{Theorem}[section]
\newtheorem{lemma}[theorem]{Lemma}
\newtheorem{proposition}[theorem]{Proposition}
\newtheorem{corollary}[theorem]{Corollary}

\theoremstyle{definition}
\newtheorem{definition}[theorem]{Definition}
\newtheorem{example}[theorem]{Example}

\theoremstyle{remark}
\newtheorem{remark}[theorem]{Remark}
\numberwithin{equation}{section}

\numberwithin{equation}{section}

\begin{document}
\title[Models]{Canonical Models for holomorphic iteration}
\author[L. Arosio]{Leandro Arosio}
\author[F. Bracci]{Filippo Bracci}
\address{L. Arosio, F. Bracci: Dipartimento Di Matematica\\
Universit\`{a} di Roma \textquotedblleft Tor Vergata\textquotedblright\ \\
Via Della Ricerca Scientifica 1, 00133 \\
Roma, Italy} \email{arosio@mat.uniroma2.it, fbracci@mat.uniroma2.it}\thanks{Supported by the ERC grant ``HEVO - Holomorphic Evolution Equations'' n. 277691}

\begin{abstract}
We construct canonical intertwining semi-models with Kobayashi hyperbolic base space for  holomorphic self-maps of complex manifolds which are univalent on some absorbing cocompact hyperbolic domain. In particular, in the unit ball  we solve the Valiron equation for hyperbolic univalent self-maps and for hyperbolic semigroups.
\end{abstract}



\maketitle

\tableofcontents

\section{Introduction}

In 1884 K\"onigs \cite{Ko} proved that given a holomorphic self-map $f$ of the unit disc $\D$ such that $f(0)=0$ and $0<|f'(0)|<1$ there exists a unique holomorphic map $\sigma:\D \to \C$ satisfying the {\sl Schr\"{o}der equation} $$\sigma\circ f = f'(0) \sigma,$$ and such that $\sigma(0)=0$, $\sigma'(0)=1$. Notice that $\cup_{m\in \N}\, f'(0)^{-m}\sigma(\D)=\C$.
Of course the same holds if the self-map $f$ admits a fixed point $p\in \D$.

If the mapping $f$ has no fixed point in $\D$, then  by the Denjoy-Wolff theorem there exists a unique point $a\in \de \D$, called the {\sl Denjoy-Wolff point of $f$}, such that the sequence of iterates $\{f^n\}$ converges uniformly on compacta to the constant map $\D \ni z\mapsto a$. Moreover, the non-tangential limit of $f$ at $a$ is $a$ and the non-tangential limit of $f'$ at $a$ is a number $\lambda_f\in (0,1]$, called the {\sl dilation} of $f$. If $\lambda_f<1$, the map $f$ is called {\sl hyperbolic} and in 1931 Valiron \cite{Va}  proved that there exists a holomorphic map $\sigma:\D \to \Ha:=\{\zeta \in \C: \Im \zeta>0\}$ which satisfies  the Schr\"{o}der equation at the boundary
\[
\sigma \circ f=\frac{1}{\lambda_f} \sigma,
\]
also called the  {\sl Valiron equation}.
The solution $\sigma$  satisfies   $\cup_{m\in \N}\, \lambda_f^{-m}\sigma(\D)=\H$ and is  essentially unique in the sense that any other  holomorphic  solution is a positive multiple of $\sigma$ (see \cite[Prop. 2.4]{BP}).

In 1979 Pommerenke and Baker  \cite{Po1,BaPo} dealt with the {\sl parabolic case} $\lambda_f=1$, proving that in such a case the {\sl Abel equation}
\[
\sigma \circ f=\sigma + 1
\]
admits a holomorphic solution $\sigma:\D \to \C$.
For any solution $\sigma$ of the Abel equation the domain  $\cup_{m\in \N}\, (\sigma(\D)-m)$ is either the whole $\C$ or  is biholomorphic to $\D$, depending on whether $f$ is  zero-step or nonzero-step.  The uniqueness of the solutions of the Abel functional equations has been investigated in \cite{BP,CDP2,CDP,P}.

These three functional equations are examples of intertwining models, and since $f$ is intertwined with linear fractional maps, they are called {\sl linear fractional models}.
There is a geometric way of approaching the problem of intertwining models which was proposed  in 1981 by  Cowen \cite{Cowen}: given a holomorphic self-map $f$ of the unit disc, one wants to find a Riemann surface $\Omega$, an automorphism $\psi$ of $\Omega$ and a holomorphic mapping $\sigma\colon \D\to \Omega$ such that  the following diagram commutes:
$$\xymatrix{\D\ar[r]^{f}\ar[d]_{\sigma}& \D\ar[d]^{\sigma}\\
\Omega\ar[r]^{\psi}& \Omega.}$$
Notice that  one can always replace $\Omega$ with its smallest domain containing $\sigma(\D)$ on which $\psi$ is an automorphism, that is $\cup_{m\in \N}\, \psi^{-m}(\sigma(\D))$, hence we assume that  $\Omega=\cup_{m\in \N}\, \psi^{-m}(\sigma(\D)).$
Cowen exploited  a categorial construction (the tail space, also known as the abstract basin of attraction) to show that such a triple $(\Omega, \psi,\sigma)$ always exists, assuming that $f$ has no fixed point $p$ with $f'(p)=0$, and moreover the intertwining mapping $\sigma$ is  univalent on some  $f$-absorbing domain of $\D$ (that is an $f$-invariant domain which eventually contains every orbit of $f$). By
 the Riemann uniformization theorem $\Omega$ is then biholomorphic either to $\C$ or  to $\H$. The last step is to determine $\psi$, and Cowen shows that, if $f$ is not an elliptic automorphism of the disc,  there are only four possible cases:
 \begin{enumerate}
 \item $\Omega=\C$ and $\psi(z)=cz$ with $0<|c|<1$,
  \item $\Omega=\H$  and $\psi(z)= sz$ with $s\in \R, 0<s<1$,
 \item $\Omega=\H$ and  $\psi(z)=z\pm 1$,
  \item $\Omega=\C$ and  $\psi(z)=z+1$.
\end{enumerate}
In case (1) the self-map $f$ has a fixed point such that $f'(p)=c$. In case (2) the self-map $f$ is hyperbolic with dilation $\lambda_f=s$. Case (3) and (4) correspond to the nonzero-step and zero-step parabolic case.
Thus this approach allows  to treat the three previous functional equations in a unique framework.

In 1996 Bourdon and Shapiro \cite{BoSh} considered the case of holomorphic self-maps of the unit disc with no interior fixed points  having some regularity at the Denjoy-Wolff point. Assuming regularity at the Denjoy-Wolff point allows to obtain better knowledge of the intertwining map and of the shape of the corresponding image domain. Such results, aside an intrinsic interest, have been applied in many ways, for instance to study dynamics, commuting maps, semigroups and properties of the associated composition operators.

In higher dimension the situation is much more complicated. In the case of a holomorphic self-map $f$ of the unit ball (or more generally a holomorphic self-map of a taut manifold)  with the origin as a unique fixed point, the classical theory of Poincar\'e-Dulac can be applied (see, {\sl e.g.}, \cite{Ar,St,RR,C-Mc,Bi-Ge,BES}) and, for instance, if all eigenvalues have modulus less than one and in absence of resonances, one obtains a holomorphic solution $\sigma\colon \B^q\to \C^q$ of the Schr\"{o}der equation $$\sigma\circ f=d f_0 \circ \sigma,$$ such that $\sigma(0)=0$ and $d\sigma_0={\sf id}$. This gives   a linear intertwining model.

In the case of $f\in \Hol(\B^q,\B^q)$ with no interior fixed points, by a result of Herv\'e \cite{herve} (see also \cite{Mc,A}) there exists a unique point $a\in \de \B^q$, called the {\sl Denjoy-Wolff point} of $f$, such that the sequence of iterates of $f$ converges uniformly on compacta to the constant map $\B^q\ni z\mapsto a$. Moreover, the {\sl dilation} $\lambda_f\in (0,1]$ can be defined (see Section \ref{sec-ball}). Again, if $\lambda_f<1$ the map $f$ is called {\sl hyperbolic}, while it is called {\sl parabolic} if $\lambda_f=1$.
In \cite{BG} the second named author and Gentili proved that if $f$ is hyperbolic and has some regularity at the Denjoy-Wolff point $a\in\de \B^q$, then there exists a holomorphic mapping $\sigma\colon \B^q\to \C^q$ such that
$$ \sigma\circ f=df_a  \circ \sigma,$$
and $\sigma$ is  regular and essentially unique in the class of holomorphic maps having some regularity at $a$.
Such type of results, assuming  similar regularity conditions at the Denjoy-Wolff point of $f$, have later been generalized both in the hyperbolic and  parabolic case by Bayart \cite{bayart3}. Bayart  and Charpentier \cite{bayart2, bayart} also performed a deep study of  intertwining models for linear fractional self-maps of the unit ball, generalizing previous results of \cite{Bi-Br} and \cite{BCD}.

If $f\in \Hol(\B^q,\B^q)$ is hyperbolic with Denjoy-Wolff point $a\in \partial B^q$,
a holomorphic solution $\sigma\colon\B^q\to \Ha$ to the Valiron equation $$\sigma\circ f=\frac{1}{\lambda_f} \sigma$$ has been found in \cite{BGP} by the second named author, Gentili and Poggi-Corradini,  with the two following assumptions: 
\begin{enumerate}
\item  the  $K-\lim_{z\to a} \frac{1-\langle f(z),a\rangle}{1-\langle z,a\rangle}$ exists ($\langle \cdot, \cdot \rangle$ denotes the standard Hermitian product in $\C^q$ and $K\hbox{-}\lim$ denotes the limit in Kor\'anyi regions, see Section \ref{sec-ball}),
\item the mapping $f$ admits a  special orbit (see Definition \ref{restricted-special}).
\end{enumerate}
 In \cite{jury} the existence of a holomorphic solution is proved by Jury without assumption (2).

In all papers mentioned above, with the exception of Cowen's construction in dimension one, the intertwining map is defined via a limiting process using suitably normalized iterates of the original map. If the self-map has no fixed point, the estimates for the convergence of such a sequence are provided essentially by the Julia-Wolff-Carath\'eodory theorem in the unit disc, and by the Rudin's generalization of such a theorem in the unit ball. However, the Rudin's Julia-Wolff-Carath\'eodory theorem gives a good control only on the non-tangential normal direction and a holomorphic self-map of the ball may not even be differentiable at its Denjoy-Wolff point. Thus, in several variables, it seems to be hopeless to control the limiting process without assuming some regularity at the Denjoy-Wolff point.

In this paper we propose a categorial construction of  intertwining models for any holomorphic self-map which is univalent on an absorbing domain, without assuming any regularity condition at the Denjoy-Wolff point.
In particular, our results apply to univalent self-maps (such as the elements of a  continuous one-parameter semigroup of holomorphic self-maps). Here we state our results for univalent self-maps, referring the reader to Section \ref{sec-model} for details and the general case. Let $X$ be a complex manifold of dimension $q$, and let $f\in\Hol(X,X)$ be a univalent map. We define a {\sl semi-model} for $f$ as a triple $(\Omega, h, \psi)$, where the {\sl base space} $\Omega$ is a complex manifold of dimension $k\in \N$, $\psi\in {\sf Aut}(\Omega)$, and $h\in \Hol(X,\Omega)$, such that $h \circ f= \psi\circ h$ and $\cup_{m\in \N}\, \psi^{-m}(h(X))=\Omega$. If $h$ is univalent, we call the semi-model a {\sl model}. Our first result shows that models exist, are essentially unique and satisfy a universal property.

\begin{theorem}\label{main1}
Let $X$ be a complex manifold and let $f\in\Hol(X,X)$ be univalent. Then there exists a model $(\Omega, h, \psi)$ for $f$. Moreover, if $(\Lambda, \ell, \phi)$ is any semi-model for $f$ then there exists a surjective holomorphic map $g: \Omega\to \Lambda$ such that $\ell=g\circ h$ and $g\circ \psi= \phi\circ g$. In particular, if $(\Lambda, \ell, \phi)$ is a model ({\sl  i.e.}, $\ell$ is univalent), then $g$ is a biholomorphism.
\end{theorem}

Such a result {\sl \`a la Cowen} exploits the abstract basin of attraction of $f$, hence the manifold $\Omega$ is constructed as an abstract complex manifold. 
The problem with this approach in several complex variables is that there is no uniformization theorem available, and so even if $X=\B^q$, the complex structure of $\Omega$ and its automorphism group may  be very complicated.

However, in case $X=\B^q$ (or more  generally $X$ is any cocompact hyperbolic manifold), one can single out a special semi-model whose  base space is a possibly lower-dimensional ball $\B^k$, and which still gives interesting informations about $f$. This is what we do in Section   \ref{kob-semi}.
We start by noticing that $\Omega$ is the growing union of  domains biholomorphic to $\B^q$, and thus we can replace the uniformization theorem with a result of Fornaess and Sibony which describes this kind of manifolds. They prove the existence of a holomorphic submersion $r\colon \Omega\to \B^k$, where $0\leq k\leq q$, which gives a foliation $\mathcal{F}$ of $\Omega$ on whose leaves the Kobayashi pseudo-metric vanishes. Hence the automorphism $\psi$ of $\Omega$ defined in Theorem \ref{main1} preserves $\mathcal{F}$, and thus induces an automorphism $\tau$ of $\B^k$ which makes the following diagram commute:

$$\xymatrix{\B^q\ar[r]^{f}\ar[d]_{h}& \B^q\ar[d]^{h}\\
\Omega\ar[r]^{\psi}\ar[d]_{r}& \Omega\ar[d]^{r}\\
\B^k\ar[r]^{\tau}&\B^k.}$$

Define $\ell:=r\circ h$. Then $(\B^k,\ell, \tau)$ is a semi-model, and as long as Kobayashi pseudo-distance is concerned, it is indistinguishable from  the model $(\Omega,h,\psi)$: indeed the holomorphic submersion $r\colon \Omega\to \B^k$ is a pseudo-distance isometry.  This semi-model is universal in the following sense: all others  semi-models for $f$ with hyperbolic base space factor through it. We call such a semi-model the {\sl canonical Kobayashi hyperbolic semi-model}.

In Section \ref{hyper-balla} we describe the canonical Kobayashi hyperbolic semi-model of a hyperbolic univalent self-map, obtaining the following main result.

\begin{theorem}\label{iper-intro}
Let $f\in \Hol(\B^q,\B^q)$ be a hyperbolic univalent self-map with Denjoy-Wolff point $a\in \de \B^q$ and dilation $\lambda_f\in (0,1)$. Then there exists  $\ell\in \Hol(\B^q,\B^k)$, where $1\leq k\leq q$, and a hyperbolic automorphism $\tau\in {\sf Aut}(\B^k)$ with  Denjoy-Wolff point $b\in \de \B^q$ and dilation $\lambda_f$ such that
\begin{enumerate}
\item $\ell\circ f =\tau \circ \ell$,
\item  $\cup_{m\in \N}\, \tau^{-m}(\ell(\B^q))=\B^k$,
\item $K\hbox{-}\lim_{z\to a}\ell(z)=b$.
\end{enumerate}
Moreover, if $\Lambda$ is a Kobayashi hyperbolic manifold, $k\in {\sf Aut}(\Lambda)$ and $\varphi\in\Hol(\B^q, \Lambda)$ satisfies $k \circ f=\varphi\circ k$ and $\cup_{m\in \N}\,\varphi^{-m}(k(\B^q))=\Lambda$, then there exists $g\in \Hol(\B^k,\Lambda)$ surjective such that $g \circ \ell=k$ and $g\circ \tau=\varphi\circ g$.
\end{theorem}

The number $k$ in Theorem \ref{iper-intro} can be computed in terms of the dynamics of $f$. Indeed, $k$ is the rank of $\lim_{m\to \infty} (f^m)^*\kappa_{\B^q}$, where $\kappa_{\B^q}$ is the Kobayashi infinitesimal metric of $\B^q$ (see Lemma \ref{climbhazard}). 

Composing the mapping $\ell\colon \B^q\to \B^k$ with a suitable projection to $\H$, we obtain a holomorphic solution to the Valiron equation in several variables.

\begin{corollary}
Let $f\in \Hol(\B^q, \B^q)$  be a univalent hyperbolic self-map with Denjoy-Wolff point $a\in \de \B^q$ and dilation $\lambda_f\in (0,1)$.  Then there exists a
holomorphic  solution  $\Theta\colon \B^q\to\H$ to the Valiron equation
\begin{equation}
\Theta\circ f=\frac{1}{\lambda_f}\Theta
\end{equation}  which satisfies $K\hbox{-}\lim_{z\to a}\Theta(z)=\infty$ and $\cup_{n\geq 0}\, \lambda_f^n(\Theta(\B^q))=\H.$
\end{corollary}

Aside the categorial construction of the canonical Kobayashi hyperbolic semi-model, the main difficulty is to prove that the automorphism $\tau\in {\sf Aut}(\B^k)$ that one constructs abstractly is indeed  hyperbolic with dilation $\lambda_f$. The problem here is that it is not clear how the dilation $\lambda_f$  behaves  under a holomorphic semiconjugation of the map $f$.
We overcome this difficulty by introducing in Section \ref{step} a dynamical invariant that we call the {\sl divergence rate}, which generalizes the dilation and is defined in terms of the Kobayashi pseudo-distance.
For a given holomorphic self-map $f$ of a complex manifold $X$, the divergence rate is the nonnegative real number defined as $$c(f):=\lim_{m\to\infty} \frac{k_X(x,f^m(x))}{m},$$ where  $x$ is any point in $X$ and $k_X$ denotes the Kobayashi pseudo-distance.  Roughly speaking, if $X$ is hyperbolic,  the number $c(f)$ measures the rate at which the orbits of $f$ leave the compact sets. In particular, if the sequence of iterates is not compactly divergent, then we have $c(f)=0$. We prove that if  $X=\B^q$ and $f$ has no fixed point,  then $c(f)=-\log \lambda_f$. The divergence rate is well-behaved with respect to holomorphic semiconjugations: if  $g\in \Hol(Y,Y)$  and there exists $h\colon X\to Y$ such that   $h\circ f=g\circ h$, then $c(f)\geq c(g)$. We also show that for the model $(\Omega, h,\psi)$   the inequality becomes an equality, that is $c(f)=c(\psi)$. Since
the model and the canonical Kobayashi hyperbolic semi-model  $(\B^k,\ell, \tau)$  are isometric with respect to the Kobayashi pseudo-distance, we finally have  $c(f)=c(\tau)$.

In order to get the regularity of the intertwining map $\ell$ at the Denjoy-Wolff point of $f$, we prove the  following generalization of the Lindel\"of theorem (see Section \ref{sec-lindelof}), which seems to be interesting by its own and not known even in dimension one.

\begin{theorem}
Let $k,q\in \N$.  Let $h:\B^q\to \B^k$ be holomorphic. Suppose that $\{z_m\}\subset \B^q$ is a sequence converging to a point  $a\in \de \B^q$.  Assume that  the sequences $ \{k_{\B^q}(z_m, z_{m+1})\}$ and $\{k_{\B^q}(z_m, \langle z_m, a\rangle a)\}$ are bounded.
If $\lim_{m \to \infty}h(z_m)=b\in \de \B^k$ then $K\hbox{-}\lim_{z\to a}h(z)=b$.
\end{theorem}

Another generalization of the previous Lindel\"of-type result in the case $b\in \B^k$ (with more assumptions  on the sequence $\{z_m\}$) is contained in Theorem \ref{circone2}.

Results similar to Theorem \ref{iper-intro} are obtained in the elliptic case (see Proposition \ref{notdivergent}) and in the parabolic case (see Section \ref{sec-parabolic}).

As we said, if  $f$ is semiconjugate to $g$, then $c(f)\geq c(g)$ and equality holds for the canonical Kobayashi hyperbolic semi-model. It is then a natural question whether  equality has to hold also for non canonical Kobayashi hyperbolic semi-models. In Section \ref{semim},  among other results on  semi-models with Kobayashi hyperbolic base space,  we obtain a negative  answer (see Proposition \ref{controes}), indeed we prove that the Abel equation can be always solved for hyperbolic self-maps of the ball.

\begin{theorem}
Let $f\in \Hol (\B^q,\B^q)$ be univalent and hyperbolic with Denjoy-Wolff point $p\in \de \B^q$ and dilation $\lambda_f$. Then there exists $\theta: \B^q \to \H$ holomorphic such that
  \begin{equation}
  \theta \circ f=\theta+1.
  \end{equation}
   Moreover, $\bigcup_{m\in \N}(\theta(\B^q)-m)=\H$.
\end{theorem}

In Section \ref{sec-ipermani} we introduce the {\sl hyperbolic self-maps} of a complex manifold $X$, namely  $f\in \Hol(X,X)$ such that $c(f)>0$. We generalize Theorem \ref{iper-intro} to such mappings.

Finally, in Section \ref{sec-semigroup}, the results are adapted to the case of continuous one-parameter semigroups of holomorphic self-maps of a complex manifold, and mainly of the unit ball.

\section{Hyperbolic steps and divergence rate}\label{step}
Let $X$ be a  complex manifold, and let $\Hol(X,X)$ denote the set of holomorphic self-maps of $X$. If $f\in \Hol(X,X)$, for $m\in \N$, we denote by $f^m:= f^{m-1}\circ f$ and $f^0={\sf id}$. Let also $k_X$ denote the Kobayashi pseudo-distance on $X$.
\begin{definition}
Let $X$ be a  complex manifold and $f\in \Hol(X,X)$. For $x\in X$ and $m\in \N$ we denote by
\[
s_m(x):=  \lim_{n\to\infty}k_X(f^{n}(x),f^{n+m}(x))
\]
the {\sl (hyperbolic)  $m$-step} of $f$ at $x$. When no confusion can arise we will write $s(x)$ instead of $s_1(x)$.
\end{definition}
\begin{remark}\label{decrescesm}
The number $s_m(x)$ is well-defined since $$k_X(f^{n+1}(x),f^{n+1+m}(x))= k_X(f(f^{n}(x)),f(f^{n+m}(x)))\leq  k_X(f^{n}(x),f^{n+m}(x))$$ for all $n\geq 0$.
\end{remark}

The following well-known lemma is due to Fekete.
\begin{lemma}
Let $(a_m)_{m\in \N}$ be a subadditive sequence of real numbers. Then  $\lim_{m\to\infty} \frac{a_m}{m}$ exists and
$$\lim_{m\to\infty} \frac{a_m}{m}=\inf_{m\in \N}\frac{a_m}{m}.$$
\end{lemma}

\begin{lemma}\label{subadd}
For all $x\in X$, the sequences $(k_X(x,f^m(x)))$ and $(s_m(x))$    are subadditive, and therefore the limits $$\lim_{m\to\infty} \frac{k_X(y,f^m(x))}{m} \quad\mbox{and}\quad \lim_{m\to\infty} \frac{s_m(x)}{m}$$ exist.
\end{lemma}
\begin{proof}
For all $n,m\geq 0$ we have $$k_X(x,f^{m+n}(x))\leq k_X(x,f^m(x))+k_X(f^m(x),f^{m+n}(x))\leq  k_X(x,f^m(x))+k_X(x,f^n(x)),$$
which proves that $(k_X(x,f^m(x)))$ is subadditive.
For all $n,m,u\geq 0$ we have
\begin{equation*}
\begin{split}
k_X(f^u(x),f^{u+m+n}(x))&\leq k_X(f^u(x),f^{u+m}(x))+k_X(f^{u+m}(x),f^{u+m+n}(x))\\&\leq  k_X(f^u(x),f^{u+m}(x))+ k_X(f^u(x),f^{u+n}(x)).
\end{split}
\end{equation*}
Taking the limit as $u\to \infty$, we have $s_{m+n}\leq s_m+s_n$.
The limits exist by Fekete's lemma.
\end{proof}

\begin{definition}\label{divrate}
Let $X$ be a   complex manifold and let $f\in \Hol(X,X)$. The {\sl divergence rate} $c(f)\in \R^+$ is defined as 
$$c(f):=\lim_{m\to \infty}\frac{k_X(x,f^m(x))}{m}.$$
\end{definition}

\begin{remark}\label{bah}
Clearly for all $k\in \N$ we have $$c(f^k)=kc(f).$$
\end{remark}

\begin{proposition}\label{base}
Let $X$ be a   complex manifold and let $f\in \Hol(X,X)$. Let $x\in X$. Then
\begin{equation}\label{cf}
c(f)=\lim_{m\to\infty} \frac{s_m(x)}{m}.
\end{equation}
Moreover, $c(f)$ does not depend on $x$, and for all $y\in X$ we have $c(f)=\lim_{m\to\infty} \frac{k_X(y,f^m(x))}{m}$.
\end{proposition}
\begin{proof}

It is clear that $k_X(x,f^m(x))\geq s_m(x)$, thus $\lim_{m\to \infty}\frac{1}{m}k_X(x,f^m(x))\geq \lim_{m\to\infty} \frac{1}{m}s_m(x)$.
We claim that for all integers $u\geq 1$,
\begin{equation}\label{bound-cf}
\lim_{m\to \infty}\frac{k_X(x, f^{um}(x))}{um}\leq \frac{s_u(x)}{u}.
\end{equation}
 Fix $u\geq 1$. Then, for all $m\geq 1$,
\begin{equation*}
\frac{k_X(x,f^{um}(x))}{um}\leq \frac{1}{um}\sum_{i=0}^{m-1} k_X(f^{ui}(x),f^{u(i+1)}(x)),
\end{equation*}
and the claim follows since the right-hand side tends to $\frac{s_u(x)}{u}$ as $m\to \infty$ by the  Ces\`aro means theorem. Therefore  we have,  for every integer $u\geq 1$,
\begin{equation*}
\lim_{k\to\infty} \frac{k_X(x,f^{m}(x))}{m}=\lim_{m\to\infty} \frac{k_X(x,f^{um}(x))}{um} \leq \frac{s_u(x)}{u}.
\end{equation*}
Taking the limit as $u\to \infty$, we obtain $\lim_{m\to \infty}\frac{1}{m}k_X(x,f^m(x))\leq \lim_{m\to\infty} \frac{1}{m}s_m(x)$, hence \eqref{cf} holds.

Finally, let $x,y\in X$. Then
\begin{equation*}
\begin{split}
k_X(x,f^m(x))&\leq k_X(x,y)+k_X(y,f^m(y))+k_X(f^m(y),f^m(x))\\&\leq k_X(y,f^m(y)) +2 k_X(x,y),
\end{split}
\end{equation*}
which proves that $c(f)$ does not depend on $x$. A similar argument gives the last statement.
\end{proof}

\begin{remark}\label{c-elliptic}
Let $X$ be a hyperbolic complex manifold and let $f\in \Hol(X,X)$. If $\{f^k\}$ is not compactly divergent, then $c(f)=0$. In particular, if $f$ has a fixed point in $X$ then $c(f)=0$.
\end{remark}

\begin{lemma}\label{c}
Let $X$ and $Y$ be  complex manifolds.
Let $f\in \Hol(X,X)$ and $g\in \Hol(Y,Y)$. Suppose there exists $h\in\Hol(X,Y)$ which satisfies $h\circ f= g\circ h$, {\sl i.e.} $f$ is semiconjugate to $g$. Then $c(f)\geq c(g).$
\end{lemma}
\begin{proof}
Since $k_{X}(x,f^m(x))\geq k_{Y}(h(x),h(f^m(x)))=k_{Y}(h(x),g^m(h(x)))$, the result follows at once.
\end{proof}

\section{Models}\label{sec-model}

\begin{definition}
Let $X$ be a complex manifold and let $f\in \Hol(X,X)$. A domain $A\subset X$ is {\sl $f$-absorbing}  if $f(A)\subset A$ and for each point $z\in X$ there exists $m\in \N$ such that $f^m(z)\subset A$.
\end{definition}

\begin{remark}
If $f\in \Hol(X,X)$ and $A\subset X$ is an $f$-absorbing domain, then for any compact subset $K\subset X$ there exists $m\in \N$ such that $f^m(K)\subset A$. Indeed, given any $z\in K$ there exists $n(z)\in \N$ such that $f^n(z)\in A$. Hence $\{f^{-n(z)}(A)\}_{z\in K}$ is an open covering of $K$. Since $K$ is compact, there exists a finite number of points $\{z_1,\ldots, z_l\}\subset K$ such that $K\subset \cup_{j=1}^l f^{-n(z_j)}(A)$ and thus, if we set $m=\max_{j=1,\ldots, l}\{n(z_j)\}$, it follows that $f^m(K)\subset A$.
\end{remark}

Note that if $f$ is univalent on $X$, then $A=X$ is an $f$-absorbing domain on which $f$ is univalent.

\begin{definition}
Let $X$ be a complex manifold and let $f\in\Hol (X,X)$. A {\sl semi-model} for $f$ is a triple  $(\Omega,h,\psi)$  where
$\Omega$ is a complex manifold, $h\colon X\to \Omega$ is a holomorphic mapping, $\psi\colon \Omega\to\Omega$ is an automorphism such that
\begin{equation}\label{uno}
h\circ f=\psi\circ h,
\end{equation}
and
\begin{equation}\label{due}
\bigcup_{n\geq 0} \psi^{-n}(h(X))=\Omega.
\end{equation}
We call the manifold $\Omega$ the {\sl base space} and the mapping $h$ the {\sl intertwining mapping}.

If there exists an $f$-absorbing domain   $A\subset X$ such that $h|_A\colon A\to \Omega$ is univalent, we call the triple $(\Omega,h,\psi)$ a {\sl model} for $f$.

Let $(\Omega,h,\psi)$ and $(\Lambda, k,\varphi)$ be two semi-models for $f$. A {\sl morphism of models} $\hat\eta\colon(\Omega,h,\psi)\to(\Lambda, k,\varphi)$ is given by
 a holomorphic  mapping $\eta\colon \Omega\to \Lambda$ such that
$$ \eta\circ h=k,$$ and $$\varphi\circ \eta=\eta\circ \psi.$$ An {\sl isomorphism of models} is a morphism of models which admits an inverse.
\end{definition}

\begin{remark}
By \eqref{due}, $h$ is constant if and only if the semi-model $(\Omega,h,\psi)$ is {\sl trivial}, that is, $\Omega$ reduces to a point and $\psi$ is  the identity. Moreover, if $(\Omega,h,\psi)$ is a   model  for $f\in \Hol(X,X)$, then $\rm{dim}(\Omega)=\rm{dim}(X)$.
\end{remark}

The next lemma shows the relations between  the mapping $f$ and  the intertwining mapping.
\begin{lemma}\label{uni-abs}
Let $X$ be a complex manifold and  let $f\in \Hol(X,X)$.
\begin{enumerate}
  \item Assume that   $f$ admits a semi-model $(\Lambda,k,\varphi)$ such that $k$ is univalent on a domain $B\subset X$. Then $f$ is univalent on $B$.
  \item Assume that   $f$ admits a model $(\Omega,h,\psi)$. Suppose that $B\subset X$ is an $f$-invariant domain such that  $f$ is univalent on $B$. Then $h$ is univalent on $B$.
  \item Assume that   $f$ admits a semi-model $(\Lambda,k,\varphi)$. If $B\subset X$ is an $f$-absorbing set, then $k(B)$ is a $\psi$-absorbing set.
\end{enumerate}
\end{lemma}
\begin{proof}
(1) Let $x,y\in B$, and $x\neq y$. Then $k(x)\neq k(y)$, which implies $\varphi(k(x))\neq \varphi(k(y))$. Thus $k(f(x))\neq k(f(y))$, which implies $f(x)\neq f(y)$.

(2) By definition there exists an $f$-absorbing domain $A$ such that $h$ is univalent on $A$. Let $x, y \in B$, and $x\neq y$. Then for all $n\in \N$ we have  $f^n(x)\neq f^n(y)$. Since $A$ is $f$-absorbing, there exists $m$ such that $f^m(x), f^m(x)\in A$. Thus $$\psi^m(h(x))=h(f^m(x))\neq h(f^m(x))=\psi^m(h(y)),$$ which implies $h(x)\neq h(y).$

(3) Finally, assume $B$ is $f$-absorbing. Let $z\in \Lambda$. Then there exists $m\in \N$ such $\varphi^m(z) \in k(X)$. Hence, there exists $x\in X$ such that $k(x)=\varphi^m(z)$. Since $B$ is $f$-absorbing, there exists $n\in \N$ such that $f^n(x)\in B$. Hence
\[
\varphi^{m+n}(z)=\varphi^n(\varphi^m(z))=\varphi^n(k(x))=k(f^n(x))\in k(B),
\]
which proves that $k(B)$ is $\varphi$-absorbing.
\end{proof}

Morphism between semi-models are very rigid, as the following lemmas show.
\begin{lemma}\label{morf-su}
Let $(\Omega,h,\psi)$ and $(\Lambda, k,\varphi)$ be two semi-models for $f$ and let $\hat \eta\colon  (\Omega,h,\psi)\to(\Lambda, k,\varphi)$ be a morphism of models. Then
$\eta\colon \Omega\to \Lambda$ is surjective.
\end{lemma}
\begin{proof}
Indeed, $$\Lambda=\bigcup_{n\geq 0}\phi^{-n}(k(X))=\bigcup_{n\geq 0}\eta(\psi^{-n}(X))=\eta(\Omega).$$
\end{proof}
\begin{lemma}\label{two-one}
Let $(\Omega,h,\psi)$ and $(\Lambda, k,\varphi)$ be two semi-models for $f$ and let $\hat\beta,\hat\eta\colon  (\Omega,h,\psi)\to(\Lambda, k,\varphi)$ be two morphisms of models. Then $\hat\eta=\hat \beta$.
\end{lemma}
\begin{proof}
Since $\eta\circ h=\beta\circ h$, it follows $$\eta\circ \psi^{-n}\circ h= \varphi^{-n}\circ \eta\circ h= \varphi^{-n}\circ\beta\circ h=\beta\circ \psi^{-n}\circ h,$$ which implies
$\eta=\beta$ on $\Omega$ by \eqref{due}.
\end{proof}
Clearly a morphism of models for which $\eta$ univalent is in fact an isomorphism of models. But something more is true:
\begin{corollary}\label{inverse}
Let $(\Omega,h,\psi)$ and $(\Lambda, k,\varphi)$ be two semi-models for $f$ and let $\hat \eta\colon  (\Omega,h,\psi)\to(\Lambda, k,\varphi)$ and $\hat\beta\colon(\Lambda, k,\varphi)\to  (\Omega,h,\psi)$ be two morphisms of models. Then $\hat\beta={\hat\eta}^{-1}$ and the two semi-models are isomorphic.
\end{corollary}
\begin{proof}
It is easy to see that $\hat{\sf id}_{\Omega}$ and ${\hat\beta}\circ \hat \eta$ are  morphism of models from  $ (\Omega,h,\psi)$ to itself. Thus ${\hat\beta}\circ \hat \eta=\hat{\sf id}_{\Omega}$ by Lemma \ref{two-one}. Similarly, $ {\hat\eta}\circ\hat\beta=\hat{\sf id}_{\Lambda}.$
\end{proof}
\begin{remark}\label{ordering}
The previous results show that there is a natural partial ordering among the isomorphism classes of semi-models for $f$, defined in the following way: $(\Omega, h,\psi)\geq (\Lambda,k,\psi)$ if  there exists a morphism of models $\hat \eta\colon  (\Omega,h,\psi)\to(\Lambda, k,\varphi)$.
\end{remark}

\begin{proposition}\label{uniq}
Assume that  $f\in \Hol(X,X)$ admits a model $(\Omega,h,\psi)$. Then $(\Omega,h,\psi)$ satisfies the following universal property: if $(\Lambda, k,\varphi)$ is a semi-model for $f$, then there exists a  morphism of models
$\hat\eta\colon(\Omega,h,\psi)\to(\Lambda, k,\varphi)$.
\end{proposition}
\begin{proof}
Let $A$ be an $f$-absorbing domain  such that $h|_A\colon A\to \Omega$ is univalent. Set $\Omega_n:= \psi^{-n}(h(A)).$ Since $h(A)$ is $\psi$-absorbing it follows that $\Omega$ is the growing union of the domains $\Omega_n$. For all $n\geq 0$ define the holomorphic mapping $\eta_n\colon \Omega_n\to \Lambda$ by
\begin{equation}\label{def-morf}
\eta_n:=\varphi^{-n}\circ k\circ h|_A^{-1}\circ \psi^n
\end{equation}
Using the definition of semi-models, it is easy to see that $\eta_m|_{\Omega_n}=\eta_n$ for all $m\geq n$. Therefore, setting $\eta(w)=\eta_n(w)$ if $w\in \Omega_n$ we  have  a holomorphic map $\eta\colon \Omega\to\Lambda$.
Using the properties of semi-models, it is also easy to see that $\eta$ defines a morphism of models $\hat\eta\colon(\Omega,h,\psi)\to(\Lambda, k,\varphi)$.
\end{proof}

This shows that if $f$ admits a model,  it is the maximum element  in the ordering of semi-models.
Since two maximum elements necessarily coincide by Corollary \ref{inverse}, we have the following.
\begin{corollary}\label{unique-model}
Let $X$ be a complex manifold. Let $f\in \Hol(X,X)$. If  $f$ admits a model then it is unique (up to isomorphisms of models). In particular, the base space of the model is unique up to biholomorphisms.
\end{corollary}

As a direct consequence of the previous corollary we have the following result on the essential uniqueness of the intertwining map once the base space and the automorphism are fixed.

\begin{corollary}
Let $X$ be a complex manifold. Let $f\in \Hol(X,X)$.  Let $(\Omega,h,\psi)$ and $(\Omega, k,\psi)$ be two models for $f$ with the same base space.  Then there exists an automorphism $\Psi$ of $\Omega$ commuting with $\psi$ such that $k=\Psi\circ h$.
\end{corollary}

By Lemma \ref{uni-abs}, in order for $f\in \Hol(X,X)$ to admit  a model, a necessary condition is that  there exists an $f$-absorbing open subset $A\subset X$ such that $f|_A$ is univalent. The following result shows that  such a condition is also   sufficient.

\begin{theorem}\label{tifa}
Let $X$ be a complex manifold. Let $f\in \Hol(X,X)$. Suppose there exists an $f$-absorbing domain $A\subset X$ such that $f|_A$ is univalent. Then there exists a model $(\Omega, h,\psi)$ for $f$.
\end{theorem}
\begin{proof}
Endow $\N$ with the discrete topology and consider the equivalence relation $\sim$ on $A\times \N$ given as follows:
$(x,m)\sim (y,n)$  if and only if there exists $u\in \N$, $u\geq m,n$ such that $f^{u-m}(x)=f^{u-n}(y)$.

Let $\Omega:= A\times \N/_\sim$ endowed with the quotient topology. Let $\pi\colon A\times \N\to \Omega$ be the natural projection and for $n\in \N$ let $h_n \colon A \to \Omega$ be defined as $h_n(x):= \pi (x,n)$. Clearly $h_n$ is  continuous and open for all $n\in \N$. Moreover each  $h_n$  is injective since $f$ and   all its iterates are injective on $A$. Thus  each  $h_n$ is a homeomorphism on its image. Set $\Omega_n:= h_n(A)$. Since  $\cup_{n\in \N} \Omega_n=\Omega$,  it follows easily that $\Omega$ is a Hausdorff, second countable, arcwise-connected topological space. Notice moreover that
\begin{equation}\label{associazione}
h_n=h_m\circ f^{m-n}|_A,\quad 0\leq n\leq m\in \N.
\end{equation}
Therefore, $\{(A, h_n)\}_{n\in \N}$ is an atlas for $\Omega$ which gives $\Omega$ a structure of complex manifold of the same dimension of $X$. The maps $h_n$ are thus univalent. Equation (\ref{associazione}) implies that $\Omega_n\subset \Omega_m$ for all $0\leq n\leq m$.

Define $$\psi(z)= h_0\circ h_1^{-1}(z),\quad z\in \Omega_1.$$ We can extend this mapping to $\Omega_2$ by setting  $$\psi(z)= h_1\circ h_2^{-1}(z),\quad z\in \Omega_2.$$
The map $\psi$ is well-defined. Indeed, by \eqref{associazione}, it follows $h_2^{-1}\circ h_1= h_1^{-1}\circ h_0$ and hence for all $z\in \Omega_1,$
$$h_1\circ h_2^{-1}(z)=h_1\circ h_2^{-1}\circ h_1\circ h_1^{-1}(z)=h_1\circ h_1^{-1}\circ h_0\circ h_1^{-1}(z)=h_0\circ h_1^{-1}(z).$$

We can then extend $\psi$ on $\Omega_n$ inductively as $\psi(z)=h_{n-1}\circ h_n^{-1}(z)$ for $z\in \Omega_n$.  It is clear that $\psi \colon \Omega\to \Omega$ is an automorphism.

Let $h:=    h_0\colon A\to \Omega$. Then by \eqref{associazione}
\begin{equation}
\label{onA}
\psi\circ h=h_0\circ h_1^{-1}\circ h_0=h_0\circ f|_A=h\circ f|_A.
\end{equation}

Now we extend $h$ to a holomorphic map, still denoted by the same name, $h: X \to \Omega$ such that $\psi\circ h=h\circ f$ on $X$. In order to do this, let $z\in X$. Then there exists $m\in \N$ such that $f^m(z)\in A$. We set
\[
h(z):=    \psi^{-m} (h_0(f^m(z))).
\]
By \eqref{onA} it is easy to see that $h$ is well defined, holomorphic and  intertwines   $f$ with $\psi$. Moreover, $h|_A=h_0\colon A\to \Omega$ is univalent.

Moreover, by the very definition of $\psi$ we have $h_n=\psi^{-n}\circ h|_A$ for all $n\geq 0$ and $\cup_{n\in \N} h_n(A)=\Omega$. Hence
\begin{equation}\label{2star}
\bigcup_{n\in \N} \psi^{-n}(h(A))=\Omega,
\end{equation}
and the proof is concluded.
\end{proof}

\begin{remark}
By Corollary \ref{unique-model} the model $(\Omega,h,\psi)$ constructed in Theorem \ref{tifa} does not depend on the set $A$.
\end{remark}

\begin{remark}
The construction in Theorem \ref{tifa} relies on the  abstract basins of attraction of a univalent self-map (see, {\sl e.g.}, \cite{Aro1,ABW,Fo-St}). A similar construction, together with the uniformization theorem, was employed in \cite{Cowen} for constructing models for holomorphic self-maps of the unit disc. In that case it is known that,  except for mappings with one fixed point in the unit disc of multiplicity strictly greater than $1$, every holomorphic self-map $f$ of the unit disc admits a simply connected $f$-absorbing domain   on which it is univalent (see \cite[Thm. 2]{Po1}). In higher dimension, even for the unit ball, the existence of such an $f$-absorbing domain is in general false (for example if the map has non maximal rank at every point).
\end{remark}

The hyperbolic geometry of the base space of a model depends on the dynamics of the map, as the following result shows.
\begin{lemma}\label{climbhazard}
Let $X$ be a  complex manifold and let $f\in \Hol(X,X)$. Assume that $f$ admits a model  $(\Omega,h,\psi)$ and let $A\subset X$ be an $f$-absorbing domain such that $h|_A$ is univalent. Then
\begin{equation}\label{gia}
h^\ast\kappa_\Omega=\lim_{m\to\infty}(f^{m})^*\kappa_{A}=\lim_{m\to\infty}(f^{m})^*\kappa_{X},
\end{equation}
and
\begin{equation}\label{barret}
h^\ast k_\Omega=\lim_{m\to\infty}(f^{m})^*k_{A}=\lim_{m\to\infty}(f^{m})^*k_{X}.
\end{equation}
In particular,  $c(f|_A)=c(\psi)=c(f).$
\end{lemma}
\begin{proof}
We give the proof of  \eqref{barret} as  \eqref{gia} can be obtained by minor modifications. By Lemma \ref{uni-abs}  the complex manifold $\Omega$ is the union of the growing sequence of the domains $\psi^{-m}(h(A))$, thus for all $z,w \in X$
\begin{equation*}
\begin{split}
k_\Omega(h(z),h(w))&=\lim_{m\to\infty}k_{\psi^{-m}(h(A))}(h(z),h(w))
=\lim_{m\to\infty}k_{h(A)}(\psi^m(h(z)),\psi^m(h(w)))\\&=\lim_{m\to\infty}k_{h(A)}(h(f^m(z)),h(f^m(w)))=\lim_{m\to\infty}k_{A}(f^m(z),f^m(w))\\
&\geq \lim_{m\to\infty}k_{X}(f^m(z),f^m(w)),
\end{split}
\end{equation*}
where the last inequality follows from the fact that $k_A\geq k_X$ since $A\subset X$.

On the other hand, since $\psi$ is an isometry for $k_\Omega$,
\begin{equation*}
\begin{split}
k_\Omega(h(z),h(w))&=\lim_{m\to\infty}k_{\Omega}(\psi^m(h(z)),\psi^m(h(w)))\\&
=\lim_{m\to\infty}k_{\Omega}(h(f^m(z)),h(f^m(w)))\leq  \lim_{m\to\infty}k_{X}(f^m(z),f^m(w)).
\end{split}
\end{equation*}

Next, denoting by $s_m(x)$ the $m$-th hyperbolic step of $f$ and by $s_m^A(x)$ the $m$-th hyperbolic step of $f|_A$, equation \eqref{barret} implies that for all $x\in A$
\begin{equation}
s_m^A(x)=k_\Omega(h(x),\psi^m(h(x))=s_m(x).
\end{equation}
Hence Proposition \ref{base} yields $c(f|_A)=c(\psi)=c(f).$
\end{proof}

\section{Canonical Kobayashi hyperbolic semi-models}\label{kob-semi}
Let $X$ be a complex manifold. We denote by $\kappa_X$ the Kobayashi pseudo-metric of $X$, and by $k_X$ the Kobayashi pseudo-distance.
Let $f\in \Hol(X,X)$. Suppose there exists an $f$-absorbing domain   $A\subset X$ such that $f|_A$ is univalent,  and assume $(\Omega,h,\psi)$ is a model for $f$. Then $\Omega$ is the growing union of domains biholomorphic to $A$. Complex manifolds which are growing union of domains biholomorphic to a given manifold have been studied by  Forn\ae ss and Sibony \cite{FS}, who proved the following result.
\begin{theorem}[Forn\ae ss-Sibony]\label{aeris}
Let $M$ be a hyperbolic complex manifold of dimension $q$ and assume that $M/{\sf Aut}(M)$ is compact. Let $\Omega$ be a complex manifold such that $\Omega=\cup_{n\in \N} \Omega_n$, where $\Omega_n$ is  a domain biholomorphic to $M$ and $\Omega_n\subset \Omega_{n+1}$ for all  $n\in \N$. Then there exist  $r\in \Hol(\Omega,M)$ and a holomorphic retract $Z$ of $M$ of dimension $k$ such that the following holds.
\begin{itemize}
\item[i)] The map $r$ has constant rank $k$ and $\kappa_\Omega=r^\ast \kappa_M$ has constant corank $q-k$.
\item[ii)]  We have $r(\Omega)=Z$ and if $k=q$, then $Z=M$.
\item[iii)] The fibers of $r$ are $(q-k)$-dimensional connected complex submanifolds of $\Omega$ which are topologically cells. Moreover $\kappa_{r^{-1}(z)}\equiv 0$ for all $z\in M$.
\item[iv)]  If $M$ is the ball $\B^q$, then $Z$ is biholomorphic to $\B^k$. If $M$ is the polydisc $\D^q$, then $Z$ is biholomorphic to $\D^k$.
\item[v)] If $k=q$, then $\Omega$ is biholomorphic to $M$.
\end{itemize}
\end{theorem}

\begin{remark}\label{retrazione}
Notice that, since $Z$ is a holomorphic retract of $M$, it follows that $k_M|_Z=k_Z$ and $\kappa_M|_Z=\kappa_Z$. From i) of Theorem \ref{aeris} it follows immediately that $r\colon \Omega\to Z$ satisfies
\begin{equation}\label{eins}
r^* \kappa_Z=\kappa_\Omega,
\end{equation} and
\begin{equation}\label{zwei}
r^* k_Z=k_\Omega.
\end{equation}
\end{remark}

We start examining the case of hyperbolic base space:

\begin{proposition}\label{hyper-model}
Let $X$ be a complex manifold. Let $f\in \Hol(X,X)$. Suppose $A\subset X$ is an $f$-absorbing hyperbolic domain such that $f|_A$ is univalent and  $A/{\sf Aut}(A)$ is compact.  Let $(\Omega,h,\psi)$ be a model for $f$. Then the following are equivalent:
\begin{enumerate}
  \item $\Omega$ is hyperbolic,
  \item the corank of $\kappa_\Omega$ is zero,
  \item $\Omega$ is biholomorphic to $A$,
\item there exists $x\in X$ such that $\lim_{m\to\infty}\kappa_X(f^m(x),(df^m)_x v)> 0$ for all $v\in T_xX\setminus\{ 0\}$,
 \item for every $x,y\in X$ we have $\lim_{m\to\infty}k_X(f^m(x),f^m(y))>0,$
   \item there exists a model for $f$ with hyperbolic base space.
\end{enumerate}
\end{proposition}
\begin{proof}
By Theorem \ref{aeris} we have that (1), (2) and (3) are equivalent.  By Lemma \ref{climbhazard}, (2) is equivalent to (4) and (5). Finally (6) and (1) are equivalent by Corollary \ref{unique-model}.
\end{proof}

Proposition \ref{hyper-model} and Lemma \ref{climbhazard} show that even if the complex structure of the base space $\Omega$ of the model depends only on the dynamics of $f|_A$, its hyperbolic geometry  is related to  the hyperbolic geometry of the whole $X$.

\begin{example}
Let $f\in \Hol(\H, \H)$ be defined by $f(\zeta)=\zeta+i$. Since $f$ extends to a univalent mapping $\tilde{f}(\zeta)=\zeta+i$ of $\C$ to itself and $\H$ is absorbing for $\tilde f$. Since $k_{\C}\equiv 0$, it follows from Proposition \ref{hyper-model} (or from the fact that every holomorphic mapping from $\C$ to a hyperbolic Riemann surface is constant) that  $\tilde{f}$  does not admit any model with hyperbolic base space, which implies that neither does $f$.
\end{example}

The next example shows that there exists a self-map $f$ of a hyperbolic manifold $X$ which admits a model whose base space is hyperbolic but is not biholomorphic to $X$.
\begin{example}
Let $\lambda>1$ and set $X:=\Ha\setminus\{i \lambda^{-n}\}_{n\in \N}$. Let $f(\zeta):=\lambda \zeta$. We can take $A:=\{\zeta\in \C: \Im \zeta>1\}$. It is easy to see that $A$ is $f$-absorbing and, since $A$ is biholomorphic to the unit disc, Proposition \ref{hyper-model} applies.  In such a case, $f|_A:A \to A$ is a ``hyperbolic self-map'' and, as we show in Section \ref{hyper-balla} (or see \cite{Cowen}), the base space $\Omega$ is biholomorphic to the unit disc. Notice that  $f$ is univalent in $X$, but $X/{\sf Aut}(X)$ is not compact.
\end{example}

\begin{remark}
Proposition \ref{hyper-model} has the following consequence. Let $X$ be a complex manifold and let $f\in \Hol(X,X)$. If $f$ admits a model $(\Omega,h,\psi)$ such that $\Omega$ is hyperbolic, then {\sl every} $f$-absorbing hyperbolic domain $A\subset X$ such that $f|_A$ is univalent and $A/{\sf Aut}(A)$ is compact is biholomorphic to $\Omega$.
\end{remark}

If the corank of $\kappa_\Omega$ is greater than zero, it is not possible to find a model whose base space is hyperbolic. Anyway, the next result shows that there exists a semi-model whose base space is hyperbolic and which is, in some sense, canonical. Such a semi-model is the maximum of all  semi-models with hyperbolic base space with respect to the partial ordering introduced in Remark \ref{ordering}.
\begin{definition}
Let $X$ be a complex manifold and let $f\in \Hol(X,X)$. Let $(\Gamma, g,\alpha)$ be a semi-model  $(\Gamma, g,\alpha)$ with $\Gamma$ hyperbolic. We say that  $(\Gamma, g,\alpha)$ is a {\sl canonical Kobayashi hyperbolic semi-model} if for any semi-model
$(\Lambda, k,\varphi )$ for $f$ such that  $\Lambda$ is hyperbolic,  there exists a morphism of models $\hat\sigma\colon (\Gamma, g,\alpha)\to (\Lambda, k,\varphi )$.
 \end{definition}

The uniqueness of  the canonical Kobayashi hyperbolic semi-model follows at once by Corollary \ref{inverse}:

\begin{proposition}\label{unico-semimodel}
Let $X$ be a complex manifold and let $f\in \Hol(X,X)$.  Suppose  $(\Gamma, g,\alpha)$ is a canonical Kobayashi hyperbolic semi-model for $f$. Then $(\Gamma, g,\alpha)$ is unique up to isomorphisms of models.
\end{proposition}

Again we have essential uniqueness  for the intertwining mapping once the base space and the automorphism are fixed.
\begin{corollary}
Let $X$ be a complex manifold and let $f\in \Hol(X,X)$. Let $(\Gamma, g,\alpha)$ and $(\Gamma, k, \alpha)$ be two canonical Kobayashi hyperbolic semi-models for $f$. Then  there exists an automorphism $\eta$ of $\Gamma$ commuting with $\alpha$ such that $g=\eta\circ k$.
\end{corollary}

Let now $X$ be a complex manifold and let $f\in \Hol(X,X)$. Suppose $A\subset X$ is an $f$-absorbing hyperbolic domain for $f$ such that $f|_A$ is univalent and $A/{\sf Aut}(A)$ is compact. Let $r\colon A\to Z$ be the map defined in Theorem \ref{aeris}.
Let $\mathcal F$ be the non-singular holomorphic foliation on $\Omega$ induced by $r\colon \Omega\to Z$.

\begin{theorem}\label{canonical}
Let  $X$ be a complex manifold and let $f\in \Hol(X,X)$. Suppose $A\subset X$ is an $f$-absorbing hyperbolic domain for $f$ such that $f|_A$ is univalent and $A/{\sf Aut}(A)$ is compact. Let $(\Omega,h,\psi)$ be a model for $f$. Then $\psi$ preserves the foliation $\mathcal{F}$ of $\Omega$, and thus   induces an automorphism $\tau\colon Z\to Z$ such that
\begin{equation}\label{cid}
\tau\circ r=r\circ \psi.
\end{equation}
The triple $(Z, r\circ h, \tau)$ is a canonical Kobayashi hyperbolic semi-model for $f$.
\end{theorem}
\begin{proof}
 Since $Z$ is hyperbolic and the Kobayashi distance along the leaves of $\mathcal F$ is identically zero by Theorem \ref{aeris}, it follows easily that $k_\Omega(x,y)=0$ if and only if $x$ and $y$ belong to the same leaf of $\mathcal F$.
 The automorphism   $\psi$ preserves  $k_\Omega$, and thus  it preserves the foliation $\mathcal{F}$: two points $x$ and $y$ of $\Omega$ belong to the same leaf of $\mathcal F$ if and only if the images $\psi(x)$ and $\psi(y)$ belong to the same leaf of $\mathcal F$.

 Let $z\in Z$. If the point  $x\in \Omega$ belongs to the leaf $r^{-1}(z)$, we define $\tau\colon Z\to Z$  by setting
\[
\tau(z):=r(\psi(x)).
\]
 By the preceding discussion it follows that $\tau$ is well defined and is bijective.

We claim that $\tau$ is a holomorphic automorphism of $Z$. Let thus $z\in Z$ and $x\in r^{-1}(z)$. By the implicit function theorem there exists a neighborhood $U$ of $z$ and a holomorphic map $\gamma\colon U\to \Omega$ such that $r\circ\gamma={\sf id}_U.$ But then $$\tau=r\circ \psi\circ\gamma$$ is holomorphic in $U$. It is then easy to see that  $(Z, r\circ h, \tau)$ is a semi-model with hyperbolic base space.

Now, let  $(\Lambda, k,\varphi)$ be a semi-model for $f$, with $\Lambda$ hyperbolic. By Theorem \ref{tifa} there exists a morphism of models $\hat\eta\colon(\Omega,h,\psi)\to (\Lambda, k,\varphi)$, that is  $\eta\in \Hol(\Omega,\Lambda)$ such that $k=\eta\circ h$ and $\eta\circ\psi=\varphi\circ \eta.$ Since $\Lambda$ is hyperbolic, the mapping $\eta$ is constant on the leaves of $\mathcal F$. Therefore, given $z\in Z$ and $x\in r^{-1}(z)$, setting
\[
\sigma(z):=\eta(x),
\]
gives a well defined mapping $\sigma\colon Z\to \Lambda$ which satisfies $\eta=\sigma\circ r$ by construction. To see that $\sigma$ is holomorphic, just notice that $\sigma=\eta \circ \gamma$ on $U$, where $U$ and $\gamma$ are given by the implicit function theorem as before.

Finally, notice that
 $$\varphi\circ\sigma\circ r=\sigma\circ r\circ\psi=\sigma\circ\tau\circ r,$$
which implies $\varphi\circ\sigma=\sigma\circ\tau.$ Hence $\sigma$ induces a morphism of models from $(Z, r\circ h, \tau)$ to $(\Lambda, k,\varphi)$, which shows that indeed $(Z, r\circ h, \tau)$ is a canonical Kobayashi hyperbolic semi-model.
\end{proof}

\begin{corollary}
Let  $X$ be a complex manifold and let $f\in \Hol(X,X)$. Suppose $A\subset X$ is an $f$-absorbing hyperbolic domain for $f$ such that $f|_A$ is univalent and $A/{\sf Aut}(A)$ is compact. Let $(Z, \ell, \tau)$ be a canonical Kobayashi hyperbolic semi-model for $f$ and assume that $1\leq \dim Z<\dim X$. Then $f$ preserves the foliation induced by $\ell$.
\end{corollary}
\begin{proof}
It follows from the fact that $\psi$ preserves the foliation $\mathcal F$ of $\Omega$.
\end{proof}

The hyperbolic step defined in Section \ref{step} turns out to be useful for understanding the nature of the canonical Kobayashi semi-model.

\begin{lemma}\label{usare}
Let $X$ be a complex manifold and $f\in \Hol(X,X)$. Assume there exists an $f$-absorbing domain $A\subset X$ such that $f|_A$ is univalent and $A/{\sf Aut}(A)$ is compact. Let $(Z, \ell, \tau)$ be a  canonical Kobayashi hyperbolic semi-model for $f$. Then $c(f)=c(\tau)$.
\end{lemma}
\begin{proof}
By Lemma \ref{climbhazard} and Remark \ref{retrazione} for all $x\in A$ we have
\begin{equation}\label{sancta}
s_m(x)=k_Z(\ell(x),\tau^m(\ell(x)).
\end{equation}
Hence, the result follows from Proposition \ref{base}.
\end{proof}

\begin{proposition}\label{puntifissi}
Let $X$ be a complex manifold and $f\in \Hol(X,X)$. Assume there exists an $f$-absorbing domain $A\subset X$ such that $f|_A$ is univalent and $A/{\sf Aut}(A)$ is compact. Let $(Z, \ell, \tau)$ be a  canonical Kobayashi semi-model for $f$. Then   $z\in Z$ is a fixed point of $\tau$ if and only if there exists $x\in X$ such that $z=\ell(x)$ and $s(x)=0$.
\end{proposition}
\begin{proof} Let $(\Omega, h, \psi)$ be a model for $f$. We can assume that the canonical Kobayashi hyperbolic semi-model $(Z, \ell, \tau)$ is given by  Theorem \ref{canonical}, with $\ell:=    r\circ h$.
Assume $z$ is a fixed point of $\tau$, then the leaf $r^{-1}(z)$ of $\mathcal{F}$ is $\psi$-invariant, and thus  $r^{-1}(z)\cap h(X)\neq\varnothing$ because $h(X)$ is $\psi$-absorbing.   If $x\in X$ is such that $h(x)\in r^{-1}(z)$, then (\ref{sancta}) shows that $s(x)=0$.
Conversely, assume that $x\in X$ satisfies $s(x)=0$,  then  (\ref{sancta}) shows that $r^{-1}(z)$ is $\psi$-invariant, that is $\tau(z)=z$.
\end{proof}

We end up this section by constructing a canonical Kobayashi hyperbolic semi-model for any holomorphic self-map of a taut manifold whose sequence of iterates is not compactly divergent.

Assume $X$ is a taut complex manifold and $f\in \Hol(X,X)$. If $\{f^m\}$  is not compactly divergent (for instance if $f$ has a periodic or a fixed point in $X$) then by \cite[Theorem 2.1.29]{A} there exists a complex submanifold $M\subset X$ (called the {\sl limit manifold} of $f$) and a holomorphic retraction $\rho: X \to M$ (which we call the {\sl canonical retraction} associated with $f$) such that $\rho$ is a limit point of $\{f^m\}$ and every other limit point of $\{f^m\}$ is of the form $\gamma\circ \rho$ for some automorphism $\gamma$ of $M$. Moreover, $f(M)=M$, $f|_M$ is an automorphism of $M$ and $\rho \circ f=f\circ \rho$. Note that  any $f$-absorbing domain $A\subset X$---if it exists---has to contain $M$. Then we can construct a canonical Kobayashi hyperbolic semi-model in this case.
\begin{proposition}\label{notdivergent}
Assume $X$ is a taut complex manifold and $f\in \Hol(X,X)$. Suppose $\{f^m\}$  is not compactly divergent and let $M$ be the limit manifold of $f$ and $\rho$ the canonical retraction associated with $f$. Then $(M,\rho, f|_M)$ is a canonical Kobayashi hyperbolic semi-model for $f$. In particular, $f$ has a model with hyperbolic base space if and only if $f$ is an automorphism of $X$.
\end{proposition}
\begin{proof}
Clearly, $(M,\rho, f|_M)$ is a Kobayashi hyperbolic semi-model for $f$. In order to prove that it is a canonical one, we prove that for any Kobayashi hyperbolic semi-model $(\Lambda,k,\v)$ there exists a morphism of models $\hat\eta\colon (M,\rho, f|_M)\to (\Lambda,k,\v)$.
Since $\rho$ is a limit point of $\{f^m\}$ there exists a subsequence $\{f^{m_n}\}$ converging to $\rho$ uniformly on compacta of $X$. Let $x,y\in X$ be such that $\rho(x)=\rho(y).$ For all $n\geq 0$ we have $$k\circ f^{m_n}=\v^{m_n}\circ k,$$ and thus
\begin{equation*}
\begin{split}
0=k_\Lambda(k(\rho(x)), k(\rho(y))&=\lim_{n\to \infty}k_\Lambda(k(f^{m_n}(x)),k(f^{m_n}(y)))\\&=\lim_{n\to \infty} k_\Lambda(\v^{m_n}(k(x)),\v^{m_n}(k(y)))=k_\Lambda(k(x),k(y)).
\end{split}
\end{equation*}
Since $\Lambda$ is hyperbolic this implies $k(x)=k(y)$. Therefore $k$ is constant along the fibers of $\rho$, hence $k\circ \rho=k$. Thus, it is easy to see that the map $\eta \colon M\to \Lambda$ defined as $\eta=k|_M$ gives the desired morphism of models.
\end{proof}

\section{Univalent self-maps of the  ball}\label{sec-ball}
\subsection{Iteration in the unit ball} In this section we focus on the model problem for {\sl univalent} self-maps of the unit ball. We briefly recall here some facts about holomorphic self-maps of the unit ball, referring the reader to \cite{A}, \cite{Ru} for details and proofs.

We use the Poincar\'e metric of the disc $\D$ with constant curvature $-1$.  Let $\B^q:=\{z\in \C^q : \|z\|<1\}$, and recall that for all $z\in \B^q$ we have
\[
k_{\B^q}(0,z)=\log \frac{1+\|z\|}{1-\|z\|},
\]
while $k_{\B^q}(z,w)=k_{\B^q}(0,T_w(z))$ where $T_w\in {\sf Aut}(\B^q)$ is any automorphism of $\B^q$ which maps $w$ to $0$.

The Siegel upper half-space $\mathbb H^q$ is defined by $$\mathbb{H}^q=\left\{(z,w)\in \C\times \C^{q-1}, \Im(z)>\|w\|^2\right\}.$$ Recall that $\mathbb H^q$ is biholomorphic to the ball $\B^q$ via the {\sl Cayley transform} $\Psi\colon \B^q\to \H^q$ defined as $$\Psi(z)=i\frac{e_1+z}{1-z_1}.$$

In  several variables the natural notion of limit at the boundary of the ball is defined using the Kor\'anyi approach regions.
 If $a\in \partial\B^q$, then the set $$K(a,R):=\{z\in \B^q: |1-(z,a)|< R(1-\|z\|)\}$$ is a {\sl Kor\'anyi region of vertex $a$ and amplitude $R> 1$} (see
\cite[Section 2.2.3]{A}). In \cite[Section 5.4.1]{Ru}   a slightly different but essentially equivalent definition is given and used.
Let $f: \B^n \to \C^m$ be a holomorphic map. We say that $f$ has {\sl $K$-limit} $L$ at
$a$---and we write $K\hbox{-}\lim_{z\to a}f(z)=L$---if for
each sequence $\{z_k\}\subset \B^q$ converging to $a$ such that
$\{z_k\}$  belongs eventually to some Kor\'anyi region of vertex $a$, we have
that $f(z_k)\to L$.

In what follows we shall also need a slightly weaker notion of limit.
\begin{definition}\label{restricted-special}
A sequence $\{z_k\}\subset \B^n$ converging to $a\in \partial\B^q$ is said to be {\sl restricted} at $a$ if and only if  $\la z_k, a\ra\to 1$ non-tangentially in $\D$, while  it is said to be {\sl special} at $a$ if and only if
\[
\lim_{k\to \infty}k_{\B^n}(z_k,\langle z_k,a\rangle a)=0.
\]

This latter condition is equivalent to $\lim_{k\to \infty}\frac{\|z_k-\la z_k,a\ra a\|^2}{1-|\la z_k,a\ra|^2}= 0$.
We say that $\{z_k\}$ is  {\sl admissible} at $a$ if it is both special and restricted at $a$.
We say that $f$ has {\sl admissible limit} (or  {\sl restricted $K$-limit})
$L$ at $a$---and we write $\angle_K\lim_{z\to a}f(z)=L$---if
for every admissible sequence $\{z_k\}\subset \B^n$ converging to $a$ we have
that $f(z_k)\to L$.
\end{definition}

One can show that
\[
K\hbox{-}\lim_{z\to a}f(z)=L\Longrightarrow \angle_K\lim_{z\to
a}f(z)=L,
\]
but the converse implication  is not true in general.

For the convenience of the reader we recall the following Lemma proved in \cite[Lemma 2.4]{BGP}.
\begin{lemma}\label{BGP}
Let $\{z_k\}$ be a sequence in $\B^q$ converging to $a\in \partial\B^q$. Then the following are equivalent:
\begin{enumerate}
\item $\{z_k\}$ is eventually contained in some Kor\'anyi region with vertex $a$,
\item $\{z_k\}$ is restricted at $a$ and the sequence $\{k_{\B^n}(z_k,\langle z_k,a\rangle a)\}$ is bounded.
\end{enumerate}
\end{lemma}

Recall that the group  of automorphisms ${\sf Aut}(\B^q)$  is made of linear fractional maps (see {\sl e.g.} \cite{A,Ru}). In particular every automorphism of $\B^q$ extends analytically to $\de \B^q$. An automorphism $\tau\in {\sf Aut}(\B^q)$ is  {\sl elliptic} if it has fixed points in $\B^q$. If $\tau$ is not elliptic, then it is  hyperbolic if it has exactly two different fixed points in $\de\B^q$, and it is parabolic if it has exactly one fixed point in $\de\B^q$.

Given a holomorphic self-map $f: \B^q\to \B^q$ with no fixed points in $\B^q$ there exists a unique  point $a\in \de \B^q$, called the {\sl Denjoy-Wolff point} of $f$, such that $\{f^n\}$ converges uniformly on compacta to the constant map $z\mapsto a$. Moreover, $K\hbox{-}\lim_{z\to a}f(z)=a$. The  {\sl  dilation} (also called the {\sl boundary dilatation coefficient}) of $f$ at its Denjoy-Wolff point $a\in \de \B^q$ is defined as
\[
\lambda_f:=\liminf_{z\to a}\frac{1-\|f(z)\|}{1-\|z\|}\in (0,1].
\]

The partition of the family of automorphisms in three different categories can be extended to a partition of the family of self-maps of the ball.
\begin{definition}
A map $f\in \Hol(\B^q, \B^q)$ is  {\sl elliptic} if it has fixed points in $\B^q$. If $f$ is not elliptic, then it is {\sl hyperbolic} if its dilation  $\lambda_f\in (0,1)$, and it is {\sl parabolic} if its dilation  $\lambda_f=1$.
\end{definition}

Let $f\in \Hol(\B^q,\B^q)$ be non-elliptic, with Denjoy-Wolff point $a$ and dilation $\lambda_f\in (0,1]$. Recall that the {\sl horosphere}  $E(a,R)$ of center $a$ and radius $R>0$ is defined as $$E(a,R):=   \left\{  z\in \B^q : \frac{|1- \langle z, a\rangle|^2}{1-\|z\|^2}< R \right\}.$$
It follows from the several variable version of Julia's Lemma (\cite[Thm. 8.5.3]{Ru} and \cite[Thm. 2.2.21]{A})  that for all $R>0$,
\begin{equation}\label{Julia}
f(E(a,R))\subset E(a, \lambda_f R).
\end{equation}

Let $e_1=(1,0,\ldots, 0)$. Since the group of automorphisms of $\B^q$ acts bi-transitively on $\partial \B^q$, in our considerations regarding boundary points $a\in \B^q$ we will often assume without loss of generality that $a=e_1$.

We shall need the following generalization to higher dimension of \cite[Prop.3.3]{CD}.

\begin{proposition}\label{limit at DW}
Let $f\in \Hol(\B^q,\B^q)$  with Denjoy-Wolff point $e_1$ and dilation $\lambda_f\in (0,1]$. Then for any admissible sequence $\{z_k\}$    which converges to $e_1$ such that $\lim_{k\to \infty}\frac{1-\langle z_k,e_1\rangle}{|1-\langle z_k,e_1\rangle|}=c$, it follows
\begin{equation}\label{special-formula}
\lim_{k\to \infty} k_{\B^q}(z_k, f(z_k))=\log \frac{|\overline{c}^2+\lambda_f|+(1-\lambda_f)}{|\overline{c}^2+\lambda_f|-(1-\lambda_f)}.
\end{equation}
In particular, if $f$ is parabolic, {\sl i.e.} $\lambda_f=1$, then for any admissible sequence $\{z_k\}$    which converges to $e_1$ we have $\lim_{k\to \infty} k_{\B^q}(z_k, f(z_k))=0$.
\end{proposition}

\begin{proof}
Let $\pi_1(z)=(z_1,0,\ldots, 0)$. First of all, we claim that $\{f(z_k)\}$ is special, namely
\begin{equation}\label{specialseq}
\lim_{k\to \infty} k_{\B^q}(\pi_1(f(z_k)), f(z_k))=0.
\end{equation}
Equation \eqref{specialseq} is equivalent to
\begin{equation}\label{special-ball}
    \lim_{k\to \infty} \frac{\|f(z_k)-\langle f(z_k), e_1\rangle e_1\|^2}{1-|\langle f(z_k),e_1\rangle|^2}=0.
\end{equation}
Since $\{z_k\}$ is admissible, by Rudin's Julia-Wolff-Carath\'eodory theorem (see \cite[Thm. 8.5.6]{Ru} or \cite[Thm. 2.2.29]{A}) it follows that
\begin{equation}\label{p1}
    \lim_{k\to \infty}\frac{\|f(z_k)-\langle f(z_k), e_1\rangle e_1\|}{|1-\langle z_k, e_1\rangle|^{1/2}}=0.
\end{equation}
Notice also that
\begin{equation}\label{stima-alto}
\begin{split}
   \liminf_{k\to \infty} \frac{1-|\langle f(z_k), e_1\rangle|}{|1-\langle z_k,e_1\rangle|}& \geq \liminf_{k\to \infty}\frac{1-\| f(z_k)\|}{\|e_1-z_k\|}\frac{\|e_1-z_k\|}{|1-\langle z_k,e_1\rangle|}\\&\geq \liminf_{k\to \infty}\frac{1-\| f(z_k)\|}{1-\|z_k\|}\frac{\|e_1-z_k\|}{|\langle e_1-z_k,e_1\rangle|} \geq \lambda_f.
\end{split}
\end{equation}
From \eqref{stima-alto} it follows that
\begin{equation}\label{stima-basso}
    \limsup_{k\to \infty} \frac{|1-\langle z_k,e_1\rangle|}{1-|\langle f(z_k), e_1\rangle|^2}\leq C<+\infty
\end{equation}
for some constant $C>0$. Hence by \eqref{p1} and \eqref{stima-basso} we have
\begin{equation*}
\begin{split}
\limsup_{k\to \infty}& \frac{\|f(z_k)-\langle f(z_k), e_1\rangle e_1\|^2}{1-|\langle f(z_k),e_1\rangle|^2}\\&=\limsup_{k\to \infty}\left(\frac{\|f(z_k)-\langle f(z_k), e_1\rangle e_1\|}{|1-\langle z_k,e_1\rangle|^{1/2}}\right)^2\cdot \frac{|1-\langle z_k,e_1\rangle|}{1-|\langle f(z_k), e_1\rangle|^2}=0,
\end{split}
\end{equation*}
and \eqref{special-ball} is proved.

Now, in order to prove \eqref{special-formula}, we notice that by \eqref{specialseq} and since $\{z_k\}$ is admissible,
\begin{equation*}
\begin{split}
\limsup_{k\to \infty} k_{\B^q}(z_k, f(z_k))&\leq \limsup_{k\to \infty}[k_{\B^q}(\pi_1(z_k), \pi_1(f(z_k)))+k_{\B^q}(\pi_1(f(z_k)), f(z_k))\\&+k_{\B^q}(\pi_1(z_k), z_k)] =\limsup_{k\to \infty}k_{\B^q}(\pi_1(z_k), \pi_1(f(z_k))).
\end{split}
\end{equation*}
Thus we are left to compute  $\limsup_{k\to \infty}k_{\B^q}(\pi_1(z_k), \pi_1(f(z_k)))$, and we show  that such a limit actually exists. Now,
$k_{\B^q}(\pi_1(z_k), \pi_1(f(z_k)))=\log \frac{1+|T_k|}{1-|T_k|}$, where, setting $\zeta_k:=\langle z_k, e_1\rangle$ and $g(z_k):=\langle f(z_k), e_1\rangle$ we have
\[
T_k:=\frac{\zeta_k -g(z_k)}{1-\overline{\zeta_k}g(z_k)}.
\]
A direct computation shows
\begin{equation}\label{Ttt}
T_k=\frac{1}{\overline{\zeta_k}+\frac{1-\zeta_k}{1-g(z_k)}\frac{\overline{1-\zeta_k}}{1-\zeta_k}}
-\frac{1}{\overline{\zeta_k}\frac{1-g(z_k)}{1-\zeta_k}+\frac{\overline{1-\zeta_k}}{1-\zeta_k}}
\end{equation}
By Rudin's Julia-Wolff-Carath\'eodory theorem, $\frac{1-g(z_k)}{1-\zeta_k}\to \lambda_f$. Moreover, since $\{\pi_1(z_k)\}$ converges to $e_1$ non-tangentially by hypothesis, it follows that $\Re c\neq 0$. Then
\[
\frac{\overline{1-\zeta_k}}{1-\zeta_k}=\frac{\overline{1-\zeta_k}}{|1-\zeta_k|}\frac{|1-\zeta_k|}{1-\zeta_k}\to \frac{\overline{c}}{c}=\overline{c}^2\neq -1.
\]
Hence, taking the limit in \eqref{Ttt} it follows that $T_k\to \frac{1-\lambda_f}{\overline{c}^2+\lambda_f}$ as $k\to \infty$. Therefore
\begin{equation*}
\limsup_{k\to \infty} k_{\B^q}(z_k, f(z_k))\leq \lim_{k\to \infty}k_{\B^q}(\pi_1(z_k), \pi_1(f(z_k)))=\log \frac{|\overline{c}^2+\lambda_f|+(1-\lambda_f)}{|\overline{c}^2+\lambda_f|-(1-\lambda_f)}.
\end{equation*}
By the contractivity property of the Kobayashi distance under holomorphic maps, for all $k$ we have $k_{\B^q}(\pi_1(z_k), \pi_1(f(z_k)))\leq k_{\B^q}(z_k, f(z_k))$. Thus we have the result.
\end{proof}

\begin{remark}\label{ammissibile1}
Let $\{z_k\}\subset \B^q$ be admissible, with limit $e_1$ and let $f\in \Hol(\B^q,\B^q$) with Denjoy-Wolff point $e_1$. Since $\{z_k\}$ eventually belongs to  Kor\'anyi region with vertex $e_1$, it follows that  $\{f(z_k)\}$ is eventually contained in a Kor\'anyi region with vertex $e_1$ (see, {\sl e.g.}, \cite[Lemma 1.10]{Br}). By Lemma \ref{BGP} it follows that  $\{f(z_k)\}$ is restricted, and
by \eqref{specialseq} we have that the sequence  $\{f(z_k)\}$ is admissible at $e_1$.
\end{remark}

\subsection{A generalization of the Lindel\"of theorem}\label{sec-lindelof}

The classical Lindel\"of theorem states that if $h\in \Hol(\D,\D)$  admits a limit $L\in\overline{\D}$ along a continuous curve in $\D$ which tends to $1$ then $h$ has non-tangential limit $L$ at $1$. Such a result has been generalized to the unit ball by $\check{\hbox{C}}$irca  (see, {\sl e.g.}, \cite[Thm. 2.2.25]{A} or \cite[Theorem 8.4.8]{Ru}) who proved that if $h\in \Hol(\B^q,\C^k)$ is  bounded and admits a limit $L\in \C^k$ along a {\sl special} continuous curve in $\B^q$ which converges to $e_1$, then $\angle_K\lim_{z\to e_1}h(z)=L$. Such a result (and its generalization for normal functions or functions which are bounded in Kor\'anyi regions) has been generalized to strongly convex domains (using the so-called Lempert projection devices) and  to strongly pseudoconvex domains by Abate \cite{A, A0, A1}, by Cima and Krantz \cite{CK} and others (we refer the reader to the paper \cite{A2} for a detailed account about the subject).

Here we prove a generalization of Lindel\"of theorem in the ball which does not use continuous curves but rather sequences. Such a result seems to be new also for the one-dimensional case. We point out that the same proof applies to the case of bounded strongly convex domains with smooth boundary, using the Lempert projection devices. We leave the details to the interested reader.

\begin{theorem}\label{circone}
Let $k,q\in \N$ and let $h\in \Hol(\B^q,\B^k)$. Suppose that $\{z_n\}\subset \B^q$ is a sequence converging to a point  $a\in \de \B^q$. Assume that  there exists $C>0$ such that for all $n\in \N,$
\begin{equation}\label{boundC}
k_{\B^q}(z_n, z_{n+1})\leq C,\quad \mbox{and}\quad k_{\B^q}(z_n, \langle z_n, a\rangle a)\leq C.
\end{equation}
If $\lim_{n\to \infty}h(z_n)=b\in \de \B^k$ then $K\hbox{-}\lim_{z\to a}h(z)=b$.
\end{theorem}

\begin{proof}
Without loss of generality, we can suppose $a=e_1$  and set $\pi(z):=\langle z, e_1\rangle e_1$.

Recall that the Kobayashi distance of $\B^q$ is convex, namely (see, {\sl e.g.}, \cite[Lemma 4.1]{KKR}):
\begin{enumerate}
  \item for all $x,y,z,w \in \B^q$ and $t\in [0,1]$ we have
  \begin{equation}\label{convex1}
  k_{\B^q}(tx+(1-t)y, tz+(1-t)w)\leq \max\{k_{\B^q}(x,z),k_{\B^q}(y,w)\};
  \end{equation}
  \item for all  $x,y\in \B^q$ and  $t,s\in [0,1]$ we have
  \begin{equation}\label{convex2}
  k_{\B^q}(tx+(1-t)y, sx +(1-s)y)\leq k_{\B^q}(x,y).
  \end{equation}
\end{enumerate}

Define the  curve $\gamma:[0,+\infty)\to \B^q$ as follows. For $t\geq 0$, let $I(t)$ denote the integer part of $t$ and $[t]:=t-I(t)\in [0,1)$. Then we set $\gamma(t):=(1-[t])z_n+[t]z_{n+1}$ if $t\in [n,n+1)$, $n\in \N$. The curve is the piecewise affine interpolation of the sequence $\{z_n\}$ and it is clearly continuous.

Notice that $\gamma(n)=z_n$ for $n\in \N$ and $\lim_{\N\ni n\to \infty}z_n=e_1$ by assumption. For $m\in \N$ let $B_m$ denote the Euclidean ball of center $e_1$ and radius $1/m$. Then for every fixed $m$, the sequence $\{z_n\}$ is eventually contained in $B_m$ and, since $B_m\cap \B^q$ is convex, the curve $\gamma$ is eventually contained in $B_m$. Hence $\lim_{t\to \infty}\gamma(t)=e_1$.

Note that for all $n\in \N$, $h(\gamma(n))=h(z_n)$, hence $\lim_{\N\ni n\to \infty}h(\gamma(n))=b$. Let $\{t_k\}\subset [0,+\infty)$ be any sequence which converges to $\infty$, and let $n_k:=I(t_k)\in \N$, so that  $t_k\in [n_k, n_k+1)$. Then
\begin{equation*}
\begin{split}
k_{\B^k}(h(\gamma(t_k)), h(\gamma(n_k)))&\leq k_{\B^q}(\gamma(t_k), \gamma(n_k))= k_{\B^q}((1-[t_k])z_{n_k}+[t_k] z_{n_k+1}, z_{n_k})\\&\stackrel{\eqref{convex2}}{\leq} k_{\B^q}(z_{n_k}, z_{n_k+1})\stackrel{\eqref{boundC}}{\leq} C.
\end{split}
\end{equation*}
Hence $\lim_{k\to \infty}h(\gamma(t_k))=b$, and by the arbitrariness of the sequence $\{t_k\}$ we have  $\lim_{t\to \infty}h(\gamma(t))=b$.

Next, define $\beta(t):=\pi(\gamma(t))$. Clearly, $\beta:[0,+\infty)\to \B^q$ is continuous and moreover, $\lim_{t\to\infty}\beta(t)=e_1$. Trivially, the curve $\beta$ is contained in $\D e_1$.  As before,  let $\{t_k\}\subset [0,+\infty)$ be any sequence which converges to $\infty$, and let $n_k:=I(t_k)\in \N$. Then
\begin{equation*}
\begin{split}
k_{\B^k}(h(\gamma(t_k)), h(\beta(t_k)))&\leq k_{\B^k}(\gamma(t_k), \beta(t_k))\\&=k_{\B^q}((1-[t_k])z_{n_k}+[t_k]z_{n_k+1}, (1-[t_k])\pi(z_{n_k})+[t_k]\pi(z_{n_k+1}))\\
&\stackrel{\eqref{convex1}}{\leq}\max\{k_{\B^q}(z_{n_k},\pi(z_{n_k})), k_{\B^q}(z_{n_k+1}, \pi(z_{n_k+1}))\}\stackrel{\eqref{boundC}}{\leq} C.
\end{split}
\end{equation*}
Hence, since  $\lim_{k\to \infty}h(\gamma(t_k))=b$, it follows that $\lim_{k\to \infty}h(\beta(t_k))=b$, and by the arbitrariness of the sequence $\{t_k\}$ we have  $\lim_{t\to \infty}h(\beta(t))=b$.

The classical Lindel\"of theorem applied to each component of the map  $\D\ni \zeta \mapsto h(\zeta e_1)\in \B^k$ implies that $\angle\lim_{\zeta\to 1}h(\zeta e_1)=b$.

If $\{w_n\}\subset \B^q$ is any sequence converging to $e_1$ eventually contained in a Kor\'anyi region of vertex $e_1$, then  by Lemma \ref{BGP} there exists $M>0$ such that $k_{\B^q}(w_n, \pi(w_n))<M$ for all $n\in \N$ and $\{\pi(w_n)\}$ converges to $1$ non-tangentially. Hence it follows that
\[
\lim_{n\to \infty}k_{\B^k}(h(w_n),h(\pi(w_n))\leq \lim_{n\to \infty} k_{\B^q}(w_n, \pi(w_n))<M,
\]
and since $\lim_{n\to \infty}h(\pi(w_n))=b$, we have $\lim_{n\to \infty}h(w_n)=b$, proving the theorem.
\end{proof}

It is easy to see that the previous proof can be adapted to the case $b\in \B^k$ assuming that the sequence $\{z_n\}$ is special and the hyperbolic distance between two consecutive elements tends to zero (in this case we only obtain  the existence of the admissible limit), namely:

\begin{theorem}\label{circone2}
Let $k,q\in \N$ and  let $h\in\Hol(\B^q,\B^k)$. Suppose that $\{z_n\}\subset \B^q$ is a {\sl special} sequence converging to a point  $a\in \de \B^q$. Assume that
\begin{equation}\label{boundC2}
\lim_{n\to \infty} k_{\B^q}(z_n, z_{n+1})=0.
\end{equation}
If $\lim_{n\to \infty}h(z_n)=b\in \overline{\B^k}$ then $\angle_K\lim_{z\to e_1}h(z)=b$.
\end{theorem}

\subsection{Canonical Kobayashi hyperbolic semi-models in the unit ball}

If  $f\colon \B^q\to\B^q$ is an elliptic univalent self-map, then by  Proposition \ref{notdivergent} it admits a canonical Kobayashi hyperbolic semi-model $(M, \rho, f|_M)$ where $M$ is biholomorphic to a ball $\B^k$ with $0\leq k\leq q$.  Since the fixed points of $f$ are contained in $M$, we have that the canonical Kobayashi hyperbolic semi-model of an elliptic self-map of $\B^q$ is either trivial or of elliptic type.

In order to deal with the non-elliptic cases, we  introduce an invariant way of dealing with the dilation of a holomorphic self-map with no interior fixed points. Since automorphisms act doubly transitively on $\de \B^n$,  we will often assume, without any loss of generality, that the Denjoy-Wolff point is the point $e_1$.

A similar result is proved in \cite[Proposition 3.4]{CD} in the case of the unit disc assuming the existence of a non-tangential orbit.
We recall that the definition of the divergence rate $c(f)$ is given in Definition \ref{divrate}.
\begin{proposition}\label{uguale}
Let $f\in\Hol(\B^q,\B^q$) with no interior fixed points.  Assume $e_1$ is the Denjoy-Wolff point  of $f$ and let  $\lambda_f\in (0, 1]$ be its dilation. Then
\begin{equation}\label{grossa}
c(f)=- \log \lambda_f.
\end{equation}
\end{proposition}
\begin{proof}
Define $z_m:=    f^m(0)$.
Thus $$\liminf_{m\to\infty}\frac{1-\|z_{m+1}\|}{1-\|z_m\|}\geq \lambda_f,$$ and
$$\liminf_{m\to\infty}(1-\|z_m\|)^{1/m}\geq \liminf_{m\to\infty}\frac{1-\|z_{m+1}\|}{1-\|z_m\|}\geq \lambda_f.$$
Hence,
 $$\lim_{m\to\infty}\frac{k_{\B^q}(0,z_m)}{m}=\lim_{m\to\infty}\log \left(\frac{1+\|z_m\|}{1-\|z_m\|}\right)^{1/m}\leq  \log (1/\lambda_f)=-\log(\lambda_f).$$
Thus if $f$ is parabolic, we have $$\lim_{m\to\infty} \frac{1}{m}k_{\B^q}{(0,f^m(0))}=0,$$
and \eqref{grossa} holds.

Assume now $f$ is hyperbolic. Let $R\in (0,1]$ and let $E(e_1, R)$ be a horosphere. Note that
\begin{equation}\label{min-oro}
\inf_{z \in E(e_1, R)} k_{\B^q}(0,z)=k_{\B^q}\left(0,(\frac{1-R}{1+R},0)\right)=-\log R.
\end{equation}
Since $0\in \partial E(e_1, 1)$, by \eqref{Julia}, it follows that $f^m(0)\in \overline{E(e_1, \lambda_f^m)}$ for all $m\in \N$. Hence by \eqref{min-oro},
$$ \frac{1}{m}k_{\B^q}(0,f^m(0))\geq  \frac{1}{m} k_{\B^q}\left(0,(\frac{1-\lambda_f^m}{1+\lambda_f^m},0)\right)=-\log\lambda_f, $$ which proves that $\lim_{m\to\infty} \frac{1}{m}k_{\B^q}(0,f^m(0))=-\log\lambda_f$.

\end{proof}

\begin{corollary}\label{diminuisco}
Let $q, k\in \N$. Let $f\in \Hol(\B^q, \B^q)$ with  no interior fixed points  and let $\lambda_f$ be its dilation. Let $g\in \Hol(\B^k, \B^k)$ with  no interior fixed points  and let $\lambda_g$ be its dilation.
Suppose that $f$ is semiconjugate to $g$, that is, suppose there exists $h\in\Hol(\B^q,\B^k)$  such that $h\circ f=g\circ h.$ Then $\lambda_f\leq \lambda_g$.

In particular, a parabolic self-map of $\B^q$  cannot be semiconjugate to a hyperbolic self-map of $\B^k$.
\end{corollary}
\begin{proof}
It follows from Proposition \ref{uguale} and Lemma \ref{c}.
\end{proof}

If the semiconjugation is the intertwining mapping of a canonical Kobayashi hyperbolic semi-model, then the inequality between dilations is in fact  an equality.

\begin{corollary}\label{nointerior}
Let $f\in \Hol(\B^q, \B^q)$  be univalent and with  no interior fixed points and let $\lambda_f$ be its dilation.  Let $(\B^k,\ell,\tau)$ be a canonical Kobayashi hyperbolic semi-model for $f$ given by Theorem \ref{canonical}. Suppose $\tau$ has no interior fixed points, and let $\lambda_\tau$ be its dilation. Then $\lambda_f=\lambda_\tau$.
\end{corollary}
\begin{proof}
By Lemma \ref{usare} we have $c(f)=c(\tau)$. The result follows from Proposition \ref{uguale}.
\end{proof}

Before considering in detail the hyperbolic and parabolic case, we  prove a regularity result for the intertwining map at the Denjoy-Wolff point:

\begin{proposition}\label{dove-va}
Let $f\in \Hol(\B^q, \B^q)$  be a  self-map with no interior fixed points  and let  $a\in \de \B^q$ be its Denjoy-Wolff point.  Suppose  $f$ is semiconjugate to a non-elliptic  $\tau\in\Aut(\B^k)$, {\sl i.e.}, there exists $\ell\in\Hol(\B^q,\B^k)$  such that $\ell\circ f=\tau\circ \ell$. Let $b\in \de \B^k$ be the Denjoy-Wolff point of $\tau$.  Then
\begin{enumerate}
  \item if $f$ is hyperbolic, $K\hbox{-}\lim_{z\to a}\ell(z)=b$,
  \item if $f$ is parabolic, $\angle_K\lim_{z\to a}\ell(z)=b$.
\end{enumerate}

\end{proposition}

\begin{proof}
(1) Assume $f$ is hyperbolic. Fix $p\in \B^q$ and let $z_n:=f^n(p)$. By \cite[Section 3.5]{BP},  any orbit of $f$ is eventually contained in a Kor\'anyi region of vertex $e_1$. By Lemma \ref{BGP} this implies that there exists $C>0$ such that for all $n\in \N$
\begin{equation}\label{Cspecial}
k_{\B^q}(z_n,\pi(z_n))\leq C.
\end{equation}
Note that $k_{\B^q}(z_n, z_{n+1})\leq k_{\B^q}(p,f(p))$ and $\lim_{n\to \infty}z_n=a$. Moreover, $\lim_{n\to \infty}\ell(z_n)=\lim_{n\to \infty}\tau^n(\ell(p))=b$, hence the result follows from Theorem \ref{circone}.

(2) Assume $f$ is parabolic. By Corollary \ref{diminuisco} we have that $\tau$ is necessarily a parabolic automorphism of $\B^k$.

Since $\tau$ is parabolic, if $y\in \overline{\B^k}$ is such that $\tau(y)=y$, then necessarily $y=b$.

Let $\{z_m\}$ be an admissible sequence which converges to $a$ and assume that $\ell(z_m)\to x\in \overline{\B^k}$. Then
\[
k_{\B^k}(\ell(z_m), \tau(\ell(z_m)))=k_{\B^k}(\ell(z_m), (\ell(f(z_m))))\leq k_{\B^q}(z_m, f(z_m)).
\]
Hence by Proposition \ref{limit at DW},
\[
\lim_{k\to \infty}k_{\B^k}(\ell(z_m), \tau(\ell(z_m)))=0.
\]
Therefore $\tau(x)=x$, which implies $x=b$.

\end{proof}

\subsection{The hyperbolic case}\label{hyper-balla}

\begin{theorem}\label{iper-modello}
Let $f\in \Hol(\B^q, \B^q)$  be a univalent hyperbolic self-map with dilation $\lambda_f\in (0,1)$ and Denjoy-Wolff point $a\in \de \B^q$.   Then  $f$ admits a canonical Kobayashi hyperbolic semi-model $(\B^k, \ell, \tau)$ of hyperbolic type with dilation $\lambda_f$. Moreover, if $b\in \de \B^k$ is the Denjoy-Wolff point of $\tau$, then $K\hbox{-}\lim_{z\to a}\ell(z)=b$.
\end{theorem}
\begin{proof}
By Theorem \ref{canonical}, $f$ admits a canonical Kobayashi hyperbolic semi-model and by Theorem \ref{aeris}, the base space is biholomorphic to $\B^k$ for some $0\leq k\leq q$. Note that  by Lemma \ref{usare} and by (\ref{grossa}), we  have  $c(\tau)=c(f)=-\log \lambda_f$,  hence $k\geq 1$ and $\tau$ is a hyperbolic automorphism of $\B^k$ with multiplier $\lambda_f$.
The last statement follows  from Proposition \ref{dove-va}.
\end{proof}

\begin{remark}\label{esplicito}
Let $f\in \Hol(\B^q, \B^q)$  be a  hyperbolic self-map with dilation $\lambda_f\in (0,1)$.   Suppose  $f$ admits a canonical Kobayashi hyperbolic semi-model $(\B^k, \ell, \tau)$ of hyperbolic type with dilation $\lambda_f$, for some $1\leq k\leq q$. By \cite[Proposition 2.2.10]{A} there exists a  $(k-1)\times (k-1)$ unitary matrix $U$  and a biholomorphism $\gamma\colon \B^k\to \H^k$ such that
$$\varphi(z):=\gamma\circ\tau\circ\gamma^{-1} (z)=\left(\frac{z_1}{\lambda_f}, \frac{1}{\sqrt{\lambda_f}}U(z')\right).$$
Defining $\sigma:=  \gamma\circ\ell$ we see that $(\H^k, \sigma,\varphi)$ is a canonical Kobayashi hyperbolic semi-model for $f$.
\end{remark}

\begin{corollary}\label{valiron}
Let $f\in \Hol(\B^q, \B^q)$  be a univalent hyperbolic self-map with Denjoy-Wolff point $a\in \de \B^q$ and dilation $\lambda_f\in (0,1)$.  Then there exists a
holomorphic  solution  $\Theta\colon \B^q\to\H$ to the Valiron equation
\begin{equation}\label{valironpalla}
\Theta\circ f=\frac{1}{\lambda_f}\Theta
\end{equation}  which satisfies $K\hbox{-}\lim_{z\to a}\Theta(z)=\infty$ and
\begin{equation}\label{riempiepalla}
 \bigcup_{n\geq 0} \lambda_f^n(\Theta(\B^q))=\H.
  \end{equation}
Moreover, any holomorphic solution $\Theta\colon \H^q\to \H$ of (\ref{valironpalla}) satisfies (\ref{riempiepalla}).
\end{corollary}
\begin{proof}
Let $(\H^k, \ell,\tau)$ be the  canonical Kobayashi hyperbolic semi-model given by Theorem \ref{iper-modello}, with the form given in Remark \ref{esplicito}. Let $\pi_1\colon \H^k\to \H$ be the projection $\pi_1(z_1,z')=z_1.$ Then $\left(\H, \pi_1\circ \ell,x\mapsto \frac{x}{\lambda_f}\right)$ is a semi-model. Thus $\Theta:=    \pi_1\circ \ell$ solves (\ref{valironpalla}).

Conversely, suppose that $\Theta\colon \B^q\to\H$ is a holomorphic mapping which solves (\ref{valironpalla}). Then $$\left( \bigcup_{n\geq 0} \lambda_f^n(\Theta(\B^q)), \Theta, x\mapsto \frac{x}{\lambda_f}\right)$$ is a Kobayashi hyperbolic semi-model for $f$.
Since $(\H^k, \ell,\tau)$ is canonical, there exists a holomorphic surjective mapping $$\sigma\colon \H^k\to \bigcup_{n\geq 0} \lambda_f^n(\Theta(\B^q))\subset \H$$ such that $\Theta=\sigma\circ \ell$, and $\sigma\circ \tau=\frac{1}{\lambda_f}\sigma.$  Then for all $z\in\H$, $$\sigma\left(\frac{1}{\lambda_f}z,0\right)=\frac{1}{\lambda_f}\sigma(z,0),$$ which by a result of Heins \cite{heins} implies that  $\sigma(z,0)=az$ for some $a>0$, which implies that $\sigma(\H^k)=\H$.
\end{proof}

In particular, in dimension $q=2$, Theorem \ref{iper-modello} allows us to describe a model for $f$.

\begin{corollary}
Let $f\in \Hol(\B^2, \B^2)$  be  univalent, with Denjoy-Wolff point $a\in \de \B^2$ and dilation $\lambda_f\in (0,1)$. Let $(\Omega, h, \tau)$ be a model for $f$. Then, up to isomorphism of models,  one and only one of the following cases is possible:
\begin{enumerate}
  \item $\Omega=\B^2$ and $\tau$ is a hyperbolic automorphism of $\B^2$ with dilation $\lambda_f$. Moreover, if $b\in \de \B^2$ is the Denjoy-Wolff point of $\tau$, then $K\hbox{-}\lim_{z\to a}h(z)=b$.
  \item $\Omega=\H \times \C$ and $$\tau(\zeta, w)=\left(\frac{\zeta}{\lambda_f}, \delta_1(\zeta) w+\delta_2(\zeta)\right),$$ where $\delta_1, \delta_2\in\Hol(\H, \C)$ and $\delta_1(\zeta)\neq 0$ for all $\zeta\in \H$. Moreover, $K\hbox{-}\lim_{z\to a}h_1(z)=\infty$.
\end{enumerate}
\end{corollary}

\begin{proof}
Let $(\B^k, \ell, \theta)$ be a canonical Kobayashi hyperbolic semi-model for $f$. By Theorem \ref{iper-modello},  $k\geq 1$. If $k=2$ we are in the first case and the statement follows  from Theorem \ref{iper-modello}. If $k=1$, by \cite[Main Theorem]{FS} the base space $\Omega$ is  biholomorphic to $\D\times \C$, and the result follows from Theorem \ref{iper-modello}.
\end{proof}

In the case of a univalent hyperbolic linear fractional map there is a complete description of the model and of the  canonical Kobayashi hyperbolic semi-model. These also provide examples for which the base space of the canonical Kobayashi hyperbolic semi-model has dimension between $1$ and $q$.

\begin{example}[Linear fractional case]
Let $f\colon \H^q\to\H^q$ be a univalent hyperbolic linear fractional map. By \cite[Proposition 2.3]{bayart} the map $f$ is conjugated by an automorphism  $h\colon \H^q\to \mathbb{H}^q$ to the map $g\colon \mathbb{H}^q\to\mathbb{H}^q$ defined as
$$ g(z_1,z')=(\lambda z_1+b, Du, Av+c),$$ where $u=(z_2,\dots, z_p),v=(z_{p+1},\dots, z_{q})$ , $1\leq p\leq q$ and
\begin{itemize}
\item[(1)] $D$ is a diagonal matrix with coefficients of absolute value $\sqrt \lambda$,
\item[(2)] $A$ is a invertible matrix such that matrix $Q:=   \lambda I -A^*A$ is hermitian positive definite,
\item[(3)] $b$ is pure imaginary and $|c|^2+ (Q^{-1}A^*c,A^*c)\leq \Im b< \lambda-1.$
\end{itemize}
Notice that $g$ is the restriction of an automorphism of $\C^q$ to $\mathbb{H}^q$ (which we still denote by $g$). Thus, if $\Omega:=     \bigcup_{n\geq 0} g^{-n}(\mathbb H^q)$, then $(\Omega,h, g)$ is a model for $f$. By  \cite[Theorem 2.5]{bayart}, we have $$\Omega=\left\{(z_1,u,v)\in \C^q:\Im(z_1)>|u|^2-\frac{\Im b}{\lambda-1}\right\}.$$ Notice that  $\Lambda:=    \left\{(z_1,u)\in \C^{p}:\Im(z_1)>|u|^2-\frac{\Im b}{\lambda-1}\right\}  $ is biholomorphic to $\H^{p}$, and that  $\Omega =\Lambda\times \C^{q-p}$. Define $r\colon \Omega \to \Lambda$ as $r(z_1,u,v):=    (z_1,u)$ and the hyperbolic automorphism $\tau$ of $\Lambda$ as $\tau(z_1,u):=    (\lambda z_1+b, Du)$. Then $(\Lambda, r\circ h, \tau)$ is a canonical Kobayashi hyperbolic semi-model for $f$.
\end{example}

\subsection{The parabolic case}\label{sec-parabolic}

\begin{theorem}\label{dicotomia}
Let $f\in \Hol(\B^q, \B^q)$  be a univalent hyperbolic self-map with dilation $\lambda_f=1$ and Denjoy-Wolff point $a\in \de \B^q$. Then
\begin{itemize}
\item[i)] either there exists $z_0\in \B^q$ such that $s(z_0)=0$, and then the canonical Kobayashi semi-model is trivial or of elliptic type,
\item[ii)] or  $s(z)>0$ for all $z\in \B^q$, and then $f$ admits a canonical Kobayashi hyperbolic semi-model of parabolic type.
\end{itemize}
Moreover, in case ii), if $(\B^k, \ell, \tau)$ is a canonical Kobayashi hyperbolic semi-model and $b\in \de \B^k$ is the Denjoy-Wolff point of $\tau$, then $\angle_K\lim_{z\to a}\ell(z)=b$.
\end{theorem}
\begin{proof}
Part i)  follows from Proposition \ref{puntifissi}. So assume that $s(z)>0$ for all $z\in \B^q$ and let $(\B^k,\ell,\tau)$ be a canonical Kobayashi semi-model. The automorphism $\tau$ has no interior fixed points, thus it has to be parabolic by Corollary \ref{nointerior}. The last statement follows from Proposition \ref{dove-va}.
\end{proof}

\begin{corollary}
Let $f:\B^q\to \B^q$ be a parabolic univalent  self-map with Denjoy-Wolff point $a\in \de \B^q$. Assume there exists a point $z\in\B^q$ with such that    $\{f^k(z)\}$ is admissible at $a$. Then $s(z)=0$, and thus the canonical Kobayashi semi-model of $f$ is either  trivial or of elliptic type.
\end{corollary}

\begin{proof}
If $\{z_k:=f^k(z)\}$ is admissible at $a$, by Proposition \ref{limit at DW} it follows that $s(z)=\lim_{k\to \infty}k_{\B^q}(z_k, f(z_k))=0$, and Theorem \ref{dicotomia} gives the result.
\end{proof}

\begin{example}[Linear fractional case]
Let $f\colon \H^q\to\H^q$ be a univalent parabolic linear fractional map. By \cite{bayart2}, $f$ is conjugated by an automorphism  $h\colon \H^q\to \mathbb{H}^q$ to the map $g\colon \mathbb{H}^q\to\mathbb{H}^q$ defined as
$$g(z_1,z')=(z_1+2i\langle u,a\rangle+2i\langle w,c\rangle+b,u+a,Dv,Aw),$$

where $u=(z_2,\dots, z_r), v=(z_{r+1},\dots,z_p), w=(z_{p+1},\dots, z_q)$ , $1\leq r\leq p\leq q$ and
\begin{itemize}
\item[(1)] $D$ is a diagonal matrix with coefficients of absolute value $1$,
\item[(2)] $A$ is invertible matrix such that the matrix $Q:= I-AA^*$ is hermitian positive definite,
\item[(3)] $\Im b -|a|^2\geq \langle Q^{-1}c,c\rangle.$
\end{itemize}
Notice that $g$ is the restriction of an automorphism of $\C^q$ to $\mathbb{H}^q$ (which we still denote by $g$). Thus, if $\Omega:=  \bigcup_{n\geq 0} g^{-n}(\mathbb H^q)$, then $(\Omega,h, g)$ is a model for $f$. By \cite[Theorem 3.1]{bayart2} we have a dichotomy for $\Omega$:
\begin{itemize}
\item[i)]  if $\Im b -|a|^2>0$ then $\Omega=\C^q$,
\item[ii)]  if $\Im b -|a|^2=0$ then $\Omega=\left\{(z_1,u,v,w)\in \C^q:\Im(z_1)>|u|^2+|v|^2\right\}.$
\end{itemize}
In case i) the canonical Kobayashi semi-model is trivial.
In case ii), notice that $c=0$ and $\Lambda:= \left\{(z_1,w)\in \C^{p}:\Im(z_1)>|u|^2+|v|^2\right\}  $ is biholomorphic to $\H^{p}$, and that  $\Omega =\Lambda\times \C^{q-p}$. Define $r\colon \Omega \to \Lambda$ as $r(z_1,u,v,w):= (z_1,u,v)$ and the parabolic automorphism $\tau$ of $\Lambda$ as $\tau(z_1,u,v):= (z_1+2i\langle u,a\rangle+b,u+a,Dv)$. Then $(\Lambda, r\circ h, \tau)$ is a canonical Kobayashi hyperbolic semi-model for $f$.

\end{example}

\begin{remark}
Looking at the previous example, we see that the dichotomy of Proposition \ref{dicotomia} has the following stronger form in the linear fractional case:
\begin{itemize}
\item[i)] if there exists $z\in \B^q$ such that $s(z)=0$, then the canonical Kobayashi hyperbolic semi-model is trivial,
\item[ii)] if $s(z)>0$ for all $z\in \B^q$, then the canonical Kobayashi hyperbolic semi-model is parabolic.
\end{itemize}
\end{remark}

\noindent{\bf Open Question:} we do not know whether the alternative (i) of Theorem \ref{dicotomia} can happen with a canonical Kobayashi semi-model  of elliptic type. Namely, we do not know whether there exists a univalent parabolic self-map of $\B^q$ with a (non trivial) canonical Kobayashi hyperbolic semi-model of elliptic type.
\subsection{Non-canonical Kobayashi hyperbolic semi-models}\label{semim}

Let $f\in \Hol(\B^q,\B^q)$ be univalent. Then, by the previous results, $f$ admits a canonical Kobayashi hyperbolic semi-model $(\B^k, \ell, \tau)$. By Proposition \ref{unico-semimodel} the canonical Kobayashi hyperbolic semi-model is essentially unique and in particular its  type  is univocally determined by $f$.

A more difficult question is to determine {\sl all} possible Kobayashi hyperbolic semi-models of $f$. Clearly, if the canonical Kobayashi hyperbolic semi-model of $f$ is trivial, then all Kobayashi hyperbolic semi-models of $f$ are trivial. However, in the non-trivial case, {\sl a priori}, the base space of a  semi-model for $f$ is not necessarily biholomorphic to a ball and the answer to the problem in its generality seems to be hopeless. Thus one can restrict the problem to determine any possible Kobayashi hyperbolic  semi-model  of $f$ whose base space is (biholomorphic to) a ball, and whether or how the type of such a semi-model  can be determined by $f$ and its dynamical properties. We have the following result which generalizes the corresponding results for the one-dimensional case in \cite{P, CDMV}.

\begin{proposition}\label{cara-semi}
Let $q,m\geq 1$. Let $f\in \Hol(\B^q,\B^q)$. Let $\tau\in{\sf Aut}(\B^m)$ and let $\eta\in \Hol(\B^q,\B^m)$ be such that $\eta\circ f=\tau\circ \eta$.
\begin{enumerate}
  \item If $f$ is elliptic then $\tau$ is  elliptic.
  \item If $f$ is parabolic then  $\tau$ is  either  elliptic or  parabolic.
  \item If $f$ is hyperbolic with dilation $\lambda_f$ then $\tau$ is either elliptic,  parabolic, or  hyperbolic  with dilation  $\lambda_\tau\geq \lambda_f$.
  \item If $f$ is hyperbolic with Denjoy-Wolff point $a\in \de \B^q$, dilation $\lambda_f$ and
  \begin{equation}\label{regul-eta}
\al_\eta(a):=\liminf_{z\to a} \frac{1-\|\eta(z)\|}{1-\|z\|}<+\infty,
\end{equation}
then $\tau$ is  hyperbolic  with dilation $\lambda_\tau=\lambda_f$.
\end{enumerate}
\end{proposition}

\begin{proof}
(1) if $f$ is     elliptic, then necessarily $\tau$ has to be elliptic, because if $z\in \B^q$ is fixed by $f$, then $\eta(z)$ is fixed by $\tau$.

(2) If $f$ is parabolic, then $\tau$ cannot be hyperbolic, by Proposition \ref{diminuisco}.

(3) It follows at once from Proposition \ref{diminuisco}.

(4) By Rudin's Julia-Wolff-Carath\'eodory's theorem (see \cite[Thm. 8.5.6]{Ru} or  \cite[Thm. 2.2.29]{A}) it follows that there exists $b\in \de \B^m$ such that $\angle_K\lim_{z\to a}\eta(z)=b$. By Remark \ref{ammissibile1}, for any admissible sequence $\{z_k\}$ converging to $a$ we have that $\{f(z_k)\}$ is an admissible sequence which converges to $a$. Hence,
if $\{z_k\}$ is an admissible sequence converging to $a$, we have that $$\tau(\eta(z_k))=\eta(f(z_k))\longrightarrow b.$$
Hence $b$ is a fixed point of $\tau$. By the chain rule for boundary dilatation coefficients (see, {\sl e.g.}, \cite{ABr}) it follows that $\al_\eta(a)\cdot \lambda_f=\lambda_\tau\cdot \al_\eta(a)$ hence the dilation of $\tau$ at $b$ is $\lambda_f$, which also implies  that $\tau$ is hyperbolic.
\end{proof}

A parabolic/hyperbolic holomorphic self-map of the ball may admit Kobayashi hyperbolic semi-models of elliptic type, as the following example shows.

\begin{example}
Let $f: \D \to \D$ be a parabolic or hyperbolic automorphism of the unit disc. Let $\Gamma$ be the cyclic group generated by $f$. Then $\D/\Gamma$ is either biholomorphic to the punctured disc $\D^\ast$  or to the annulus $A=\{\zeta\in \C: r<|\zeta|<1\}$ for some $r\in (0,1)$. Let $\rho:\D\to \D/\Gamma$ be the covering map. Let $g\in {\sf Aut}(\D)$ be a hyperbolic automorphism such that $g(\{|\zeta|\leq r\})$ is contained in $\D\cap\{\Re \zeta >0\}$. Then define $\tilde{\psi}: \D/\Gamma \to \D$ as $\tilde{\psi}(\zeta):=g(\zeta)^2$. Note that $\tilde{\psi}$ is holomorphic and surjective. Now, let $\psi:= \tilde{\psi}\circ \rho: \D\to \D$. Then $\psi$ is a holomorphic surjective map and  $\psi \circ f=\tilde{\psi}\circ \rho \circ f=\tilde{\psi}\circ \rho=\psi$. Hence, $(\D, \psi, {\sf id})$ is a Kobayashi hyperbolic semi-model for $f$ of elliptic type.
\end{example}

A univalent hyperbolic map {\sl always}  admits a Kobayashi hyperbolic semi-model of parabolic type. The following result seems to be new also for the one dimensional case. We wish to thank Pavel Gumenyuk for the construction of the map $\tilde{\psi}$ in the proof.

\begin{proposition}\label{controes}
Let $f\in \Hol(\B^q,\B^q)$ be a univalent hyperbolic self-map of $\B^q$. Then $f$ admits a Kobayashi hyperbolic semi-model of parabolic type with base space the unit disc.
\end{proposition}

\begin{proof}
Let $(\H, \Theta,w\mapsto \lambda w)$, $\lambda>1$, be the   Kobayashi hyperbolic semi-model for $f$ given in Corollary \ref{valiron}.

In order to construct such a semi-model, we construct a {\sl surjective} holomorphic map $\psi:\H \to \H$ such that $\psi(\lambda w)=\psi(w)+1$. Note that if $\psi:\H \to \H$ satisfies $\psi(\lambda w)=\psi(w)+1$ then   $\psi(\lambda^m w)=\psi(w)+m$ for any $m\in \Z$ (not only for $m\in \N$). Hence, once such a map is constructed, setting $\eta:=\psi\circ \Theta$, it follows that $\eta:\B^q \to \H$ is holomorphic, $\eta(f(z))=\eta(z)+1$ and
\[
\bigcup_{n\in \N} (\eta(\B^q)-n)=\bigcup_{n\in \N}(\psi(\lambda^{-n}(\Theta(\B^q))))=\psi(\H)=\H.
\]
Therefore, $(\H, \eta, w\mapsto w+1)$ is a Kobayashi hyperbolic semi-model of parabolic type.

In order to construct $\psi$, let $\Gamma\subset {\sf Aut}(\H)$ be the cyclic subgroup generated by $w\mapsto \lambda w$, and let $\Gamma'\subset {\sf Aut}(\H)$ be the cyclic subgroup generated by $w\mapsto w+1$. Then $\H/\Gamma$ is biholomorphic to an annulus $A:=\{\zeta\in \C: r<|\zeta|<1\}$ for some $r\in (0,1)$, while $\H/\Gamma'$ is biholomorphic to the punctured disc $\D^\ast$. Let $\theta: \H\to A$  be the covering map. Assume $\tilde{\psi}\colon A\to \D^\ast$ is a surjective holomorphic map such that $\tilde{\psi}_\ast\colon\Pi_1(A)\to \Pi_1(\D^\ast)$ is the identity map.  Then $\tilde{\psi}\circ \theta$ has a lifting $\psi: \H \to \H$ such that $\psi(\lambda w)=\psi(w)+p$ for some $p\in \Z$. Since $\tilde{\psi}_\ast={\sf id}$, it follows that $p=1$, and since $\tilde{\psi}$ is surjective and $\psi(\lambda^m w)=\psi(w)+m$ for any $m\in \Z$, it is easy to see that $\psi$ has to be  surjective as well.

Therefore we are left to construct such $\tilde{\psi}$. Let $\epsilon>0$ and let $M_\epsilon:=\{\zeta\in \C: \Re \zeta>0\}\cup \{\zeta\in \D: |\Im \zeta|<\epsilon\}$. Then let $A_{\epsilon,a}:=M_\epsilon\setminus [a,1]$ for some $a<1$. If $\epsilon<<1$,  there exists $a\in (0,1)$ such that $A_{\epsilon,a}$ has the same modulus of $A$, and therefore there exists   a biholomorphic map $g: A \to A_{\epsilon,a}$. Now, let us consider $A\ni w\mapsto g(w)^2$. The image of $A$ under such a map is  given by $\C\setminus ((-\infty, -\epsilon^2]\cup \{1\})$. Let $h:\D \to \C\setminus (-\infty, -\epsilon^2]$ be the Riemann map such that $h(0)=1$, $h'(0)>0$. Define $\tilde{\psi}(w):=h^{-1}(g(w)^2)$. By construction $\tilde{\psi}: A\to \D^\ast$ is surjective, and it is easy to see that $\tilde{\psi}_\ast={\sf id}$.
\end{proof}

\subsection{A generalization to hyperbolic maps of complex manifolds}\label{sec-ipermani}

The results obtained in the case of hyperbolic self-maps of $\B^q$ can be generalized to any complex  manifold, once we define what a hyperbolic self-map is in this contest.

\begin{definition}\label{defiperbolico}
Let $X$ be a complex manifold. Let $f\in \Hol(X,X)$. We say that $f$ is {\sl hyperbolic} if $c(f)>0$.
\end{definition}

Putting together the results of the previous sections, it is easy to prove the following results.

\begin{theorem}
Let $X$ be a complex manifold of dimension $q$. Let $f\in \Hol (X,X)$ be hyperbolic. Assume there exist an $f$-absorbing  domain $A\subset X$ biholomorphic to $\B^q$ such that   $f|_A$ is univalent. Then
it admits a canonical Kobayashi hyperbolic semi-model $(\B^k,\ell,\tau)$, where $1\leq k\leq q$, and where  $\tau\in {\sf Aut}(\B^k)$ is hyperbolic with dilation $e^{-c(f)}$.
\end{theorem}
\begin{theorem}
Let $X$ be a complex manifold of dimension $q$. Let $f\in \Hol (X,X)$ be hyperbolic. Assume there exist an $f$-absorbing  domain $A\subset X$ biholomorphic to $\B^q$ such that   $f|_A$ is univalent.  Then:
\begin{enumerate}
  \item there exists $\sigma: X\to \H$ holomorphic such that
  \begin{equation}\label{Valiron}
  \sigma \circ f=e^{c(f)}\sigma
  \end{equation} Moreover, $\bigcup_{n\in \N}e^{-n c(f)}\sigma(X)=\H$.
  \item there exists $\theta: X \to \H$ holomorphic such that
  \begin{equation}\label{Abel}
  \theta \circ f=\theta+1.
  \end{equation}
   Moreover, $\bigcup_{n\in \N}(\sigma(X)-n)=\H$.
\end{enumerate}
\end{theorem}

\section{Semigroups}\label{sec-semigroup}

In this section we apply our results to semigroups of holomorphic self-maps. For details about semigroups we refer the reader to \cite{A} and \cite{RS}.

\begin{definition}
Let $X$ be a complex manifold.
A {\sl (one-parameter) semigroup} $(\phi_t)$ of holomorphic self-maps of $X$ is a continuous homomorphism from the additive
semigroup $\R^+$ of non-negative real numbers into $\Hol(X,X)$  endowed with the
compact-open topology.
\end{definition}
\begin{remark}
If $(\phi_t)$ is a  semigroup of holomorphic self-maps in $X$, then $\phi_t: X\to X$ is a univalent mapping for all $t\geq 0$ (see, {\sl e.g.}, \cite[Proposition 2.5.18]{A}).
\end{remark}

We need the following continuous-time version of the Fekete lemma (see for example \cite[Theorem 16.2.9]{kuczma}). 
\begin{theorem}\label{pippo}
Let $f\colon \R^+\to \R$ be a measurable subadditive function. Then $\lim_{t\to \infty} \frac{1}{t}f(t)$ exists and 
$$\lim_{t\to \infty} \frac{f(t)}{t}=\inf_{t\in\R^+} \frac{f(t)}{t}.$$
\end{theorem}

\begin{lemma}\label{coeff-semi}
Let $X$ be a complex manifold. Let $(\phi_t)$ be a  semigroup of holomorphic self-maps on $X$. 
Then for all $x\in X$ we have that $\lim_{t\to\infty} \frac{k_X(x,\phi_t(x))}{t}$ exists, and satisfies
$$c(\phi_1)=\lim_{t\to\infty} \frac{k_X(x,\phi_t(x))}{t}=\inf_{t\in \R^+}\frac{k_X(x,\phi_t(x))}{t}.$$
Moreover for all $t\geq 0$ we have $c(\phi_t)=t c(\phi_1)$.
\end{lemma}
\begin{proof}
The function $\R^+\ni t\mapsto k_X(x,\phi_t(x))$ is continuous and thus measurable. Moreover
\[
k_X(x,\phi_{t+s}(x))\leq k_X(x,\phi_t(x))+ k_X(\phi_t(x),\phi_{t+s}(x))\leq  k_X(x,\phi_t(x))+  k_X(x,\phi_s(x)),
\]
hence Theorem \ref{pippo} yields the result.
\end{proof}

\begin{definition}
Let $X$ be a complex manifold and let $(\phi_t)$ be a  semigroup of holomorphic self-maps of $X$. The semigroup is called {\sl hyperbolic} if $c(\phi_1)>0$.
\end{definition}

\begin{remark}
By Lemma \ref{coeff-semi}, a semigroup $(\phi_t)$ is hyperbolic if and only if for all $t>0$ the map $\phi_t$ is hyperbolic.
\end{remark}

For semigroups, the definition of models is as follows.

\begin{definition}
Let $X$ be a complex manifold. Let $(\phi_t)$ be a  semigroup of holomorphic self-maps of $X$. A {\sl semi-model} for $(\phi_t)$ is a triple  $(\Omega,h,\psi_t)$  where
$\Omega$ is a complex manifold, $h\in \Hol (X,\Omega)$ is a holomorphic mapping, $(\psi_t)$ is a  group of automorphisms of $X$ such that
\begin{equation}\label{unosemigroup}
h\circ \phi_t=\psi_t\circ h,\quad t\geq 0
\end{equation}
and
\begin{equation}\label{duesemigroup}
\bigcup_{t\geq 0} \psi_{-t}(h(X))=\Omega.
\end{equation}

We call the manifold $\Omega$ the {\sl base space} and the mapping $h$ the {\sl intertwining mapping}. If $h\in\Hol(X,\Omega)$ is univalent we call the triple $(\Omega,h,\psi_t)$ a {\sl model} for $(\phi_t)$.

 Let $(\Omega,h,\psi_t)$ and $(\Lambda, k,\varphi_t)$ be two semi-models for $(\phi_t)$. A {\sl morphism of models} $\hat\eta\colon(\Omega,h,\psi_t)\to(\Lambda, k,\varphi_t)$ is given by
 $\eta\in\Hol(\Omega,\Lambda)$ such that
$$ \eta\circ h=k,$$ and $$\phi_t\circ \eta=\eta\circ \psi_t,\quad t\geq 0.$$ An {\sl isomorphism of models} is a morphism which admits an inverse.
\end{definition}

The following result is analogous to Theorem \ref{tifa} and Proposition \ref{uniq} and we omit the proof.

\begin{proposition}
Let $X$ be a complex manifold and let $(\phi_t)$ be a  semigroup of holomorphic self-maps of $X$. Then there exists a model $(\Omega,h,\psi_t)$ for $(\phi_t)$.
The model $(\Omega,h,\psi_t)$ satisfies the following universal property: if $(\Lambda, h, \varphi_t)$ is a semi-model for $(\phi_t)$, then there exists a morphism of models $\hat\eta\colon (\Omega,h,\psi_t)\to (\Lambda, h, \varphi_t)$.
\end{proposition}
\begin{remark}
Let $f\in \Hol(X,X)$ be univalent. Then by Theorem \ref{tifa} it admits a model $(\Omega,h,\psi)$. By the uniqueness  up to isomorphisms of models and the previous proposition, it follows that $f$ is embeddable in a semigroup of holomorphic self-maps of $X$ if and only if $\psi$ belongs to a group of automorphisms of $\Omega$.
\end{remark}

The results in Section \ref{kob-semi} can be easily adapted to semigroups. We state here without proof the result on the existence of the canonical Kobayashi hyperbolic semi-model for semigroups in case $X/{\sf Aut}(X)$ is compact (similar results hold in case the semigroup has an absorbing open subset $A\subseteq X$  such that $A/{\sf Aut}(A)$ is compact):

\begin{theorem}\label{canonicalsemigroup}
Let $X$ be a complex manifold such that $X/{\sf Aut}(X)$ is compact. Let $(\phi_t)$ be a  semigroup of holomorphic self-maps of $X$. Let  $(\Omega,h,\psi_t)$ be a model for $(\phi_t)$. Let $r\colon A\to Z$ be the map defined in Theorem \ref{aeris} and let $\mathcal F$ be the non-singular holomorphic foliation on $\Omega$ induced by $r\colon \Omega\to Z$. Then $(\psi_t)$ preserves the foliation $\mathcal{F}$ of $\Omega$, and thus   induces a group of  automorphisms $(\tau_t\colon Z\to Z)$ such that
\begin{equation}
\tau_t\circ r=r\circ \psi_t, \quad t\geq 0.
\end{equation}
The triple $(Z, r\circ h, \tau_t)$ is a canonical Kobayashi hyperbolic semi-model for each $\phi_t$, $t\geq 0$, and it is called the {\sl canonical Kobayashi hyperbolic semi-model for  $(\phi_t)$}.
\end{theorem}

\begin{proposition}\label{invariante-fol}
Let $X$ be a complex manifold of dimension $q$ such that $X/{\sf Aut}(X)$ is compact. Let $(\phi_t)$ be a semigroup  of holomorphic self-maps of $X$, and assume  $(Z,\ell,\tau_t)$ is a canonical Kobayashi semi-model for $(\phi_t)$.  Set $k:= {\rm dim}\ Z$, and  assume there exists  $z\in Z$ such that $\tau(z)=z$.
Then  $\ell^{-1}(z)$ is  a  connected  complex submanifold  of $X$   of dimension $q-k$  which is invariant by $(\phi_t)$.
\end{proposition}
\begin{proof}
Up to isomorphisms of semi-model, we can assume that the canonical Kobayashi semi-model for $(\phi_t)$ is given by Theorem \ref{canonicalsemigroup}, with $\ell=r \circ h$.

Since  $\tau(z)=z$, the set $\ell^{-1}(z)$ is clearly invariant under $(\phi_t)$.
We have  $\ell^{-1}(z)=h^{-1}(r^{-1}(z))$, and by Theorem \ref{aeris}.iii) the fiber $r^{-1}(z)$ is a connected $(q-k)$-dimensional complex submanifold of $\Omega$ such that $h(X)\cap r^{-1}(z)\neq\emptyset$. Thus $\ell^{-1}(z)$ is  a   complex submanifold  of $X$   of dimension $q-k$. Consider the semigroup of holomorphic self-maps $(g_t:= \phi_t|_{\ell^{-1}(z)}\colon \ell^{-1}(z)\to \ell^{-1}(z))$. Then $(r^{-1}(z),h|_{\ell^{-1}(z)}, \psi_t|_{r^{-1}(z)})$ is a model for $(g_t)$, and $$\bigcup_{t\geq 0}\psi_{-t}(h(\ell^{-1}(z)))=r^{-1}(z).$$  Assume by contradiction that $\ell^{-1}(z)$ is not connected. Then there exists a partition $\ell^{-1}(z)=U\cup V$ of $\ell^{-1}(z)$ in two  disjoint open nonempty subsets.  Since the $(\phi_t)$-orbit of every point is connected, $U$ and $V$ are  invariant by $(\phi_t)$. But then $$r^{-1}(z)=\bigcup_{t\geq 0}\psi_{-t}(h(U))\cup
\bigcup_{t\geq 0}\psi_{-t}(h(V))$$ is a partition of $r^{-1}(z)$ in two disjoint open nonempty subsets, which gives a contradiction.
\end{proof}

\begin{corollary}
Let $X$ be a complex manifold of dimension $q$ such that $X/{\sf Aut}(X)$ is compact. Let $(\phi_t)$ be a semigroup of holomorphic self-maps of $X$. Let $(Z, \ell, \tau)$ be  a canonical Kobayashi hyperbolic semi-model and assume  $\tau={\sf id}_Z$. Let $k:= \dim Z$.
Then $X$ admits a $(q-k)$-dimensional non-singular holomorphic  foliation $\mathcal{G}$ whose leaves are  $(\phi_t)$-invariant.
\end{corollary}
\begin{proof}
It follows directly from Proposition \ref{invariante-fol}, with  the foliation $\mathcal{G}$  defined by the fibration $\ell\colon X\to Z$.
\end{proof}

\subsection{Semigroups in the unit ball}

Recall the following result from \cite{AS}, see also \cite[Thm. 2.5.27]{A} and \cite[Thm. A.1]{BCD}
\begin{theorem}\label{fix-one}
Let $(\phi_t)$ be a semigroup of holomorphic self-maps of $\B^q$. Then
\begin{enumerate}
  \item if there exists $z_0\in \B^n$ such that $\phi_1(z_0)=z_0$ then there exists $z_1\in \B^n$ such that $\phi_t(z_1)=z_1$ for all $t\geq 0$ and $c(\phi_t)=0$ for all $t\geq 0$.
  \item If $\phi_1$ has Denjoy-Wolff point $p\in \de \B^q$ and dilation $e^{-\lambda}$ for some $\lambda\geq 0$, then $\phi_t$ has Denjoy-Wolff point $p\in \de \B^q$ and dilation $e^{-\lambda t}$ for all $t>0$.
\end{enumerate}
\end{theorem}

\begin{definition}
A semigroup $(\phi_t)$  of holomorphic self-maps of $\B^q$  is {\sl elliptic} (resp. {\sl hyperbolic}, resp. {\sl parabolic}) provided $\phi_1$ is elliptic (resp. hyperbolic, resp. parabolic) in $\B^q$. As customary, in case $(\phi_t)$ is not elliptic, and $\phi_1$ has Denjoy-Wolff point $p\in \de \B^q$ and dilation $e^{-\lambda}$, we call $p$ the {\sl Denjoy-Wolff point of $(\phi_t)$} and $(e^{-\lambda t})$ the {\sl dilation} of $(\phi_t)$.
\end{definition}

\begin{remark}
Let $(\phi_t)$ be a non-elliptic semigroup of holomorphic self-maps of $\B^q$ and for $t>0$ let $\lambda_t$ denote the dilation of $\phi_t$ at the Denjoy-Wolff point $p\in \de \B^q$. By Lemma \ref{coeff-semi} and Proposition \ref{uguale}, it follows directly that $\lambda_t=e^{-c(\phi_t)}=e^{-t \lambda}$ where $\lambda=c(\phi_1)\geq 0$.
\end{remark}

A canonical Kobayashi hyperbolic semi-model $(\B^k, \ell, \tau_t)$ for a semigroup $(\phi_t)$ of holomorphic self-maps of $\B^q$ is called {\sl trivial}, of {\sl elliptic} type, of {\sl hyperbolic} type or of {\sl parabolic} type provided $(\B^k, \ell, \tau_1)$ is trivial, of elliptic type, of hyperbolic type or of parabolic type.

The results in Section \ref{sec-ball} can be easily adapted to the semigroup case. We state here without proof a general result:

\begin{theorem}
Let $(\phi_t)$ be a semigroup of holomorphic self-maps of $\B^q$.
\begin{enumerate}
  \item If $(\phi_t)$ is elliptic then it admits either a trivial or a  canonical Kobayashi hyperbolic semi-model of elliptic type. The latter case happens if and only if there exists a complex affine submanifold $V$ in $\C^q$ such that $W:=V\cap \B^q$ is invariant under $(\phi_t)$ and   $\phi_t|_W \in {\sf Aut}(W)$ for all $t\geq 0$.
  \item If $(\phi_t)$ is hyperbolic with Denjoy-Wolff point $a\in \de \B^q$ and dilation $(e^{-\lambda t})$, $\lambda>0$, then it admits a canonical  Kobayashi hyperbolic semi-model $(\B^k, \ell, \tau)$, $1\leq k\leq q$ of hyperbolic type  with  dilation $(e^{-\lambda t})$. Moreover, if $b\in \de \B^k$ is the Denjoy-Wolff point of $(\tau_t)$, then $K\hbox{-}\lim_{z\to a}\ell(z)=b$.
  \item If $(\phi_t)$ is parabolic with Denjoy-Wolff point $a\in \de \B^q$ then it admits either a trivial, or of elliptic type, or of parabolic type canonical Kobayashi semi-model. The latter case happens if and only if $s(\phi_1(z))>0$ for all $z\in \B^n$. Moreover, in such a case, if $(\B^k, \ell, \tau_t)$ is a canonical Kobayashi hyperbolic semi-model of parabolic type and $b\in \de \B^k$ is the Denjoy-Wolff point of $(\tau_t)$, then $\angle_K\lim_{z\to a}\ell(z)=b$.
\end{enumerate}
\end{theorem}

In particular, for   hyperbolic semigroups of the unit ball it is always possible to solve the Valiron equation:

\begin{corollary}
Let $f\colon \B^q\to \B^q$ be a hyperbolic semigroup of holomorphic self-maps with Denjoy-Wolff point $a\in \de \B^q$ and dilation  $(e^{-\lambda t})$ for some $\lambda>0$. Then the {Valiron equation}
\begin{equation}\label{Val-semi}
\Theta\circ f_t=e^{\lambda t}\Theta,\quad t\geq0
\end{equation}
has a holomorphic solution $\Theta\colon \B^q\to\H$ such that $K\hbox{-}\lim_{z\to a} \Theta(z)=\infty$ and
\begin{equation}\label{riempi-semi}
\bigcup_{t\geq 0} e^{-\lambda t}\Theta(\B^q)=\H.
\end{equation}
Moreover, any holomorphic solution $\Theta\colon \H^q\to \H$ of \eqref{Val-semi} satisfies \eqref{riempi-semi}.
\end{corollary}

It is interesting to note that, contrarily to the case of univalent self-mappings (see Proposition \ref{controes}), a hyperbolic semigroup of $\B^q$ does not always admit a Kobayashi hyperbolic semi-model with base space $\B^k$ and of parabolic type.
In order to show that, we need a preliminary lemma.

\begin{lemma}\label{inj}
 Let $c,d>0$ and $\varepsilon =\pm 1$.
\begin{enumerate}
\item If $\sigma\in \Hol(\H,\H)$ is such that $$\sigma(e^{ct}z)=\sigma(z)+ \varepsilon t,\quad \forall t\geq 0,$$  then $\sigma$ is univalent.
\item If $\sigma\in \Hol(\H,\H)$ is such that $$\sigma(e^{ct}z)=e^{dt}\sigma(z),\quad \forall t\geq 0,$$ then $\sigma$ is univalent.
\end{enumerate}
\end{lemma}
\begin{proof}
(1) Let $\theta\in (0,\pi)$ and let $L_\theta:=\{re^{i\theta}: r>0\}$. For $p>0$ let $H_p:=\{r+ip: r\in \R\}$. Then $\sigma(L_\theta)=H_{\Im \sigma (e^{i\theta})}$ and $\sigma$ is univalent on $L_\theta$. Assume there exists $w_0\neq w_1\in \H$ such that $\sigma(w_0)=\sigma(w_1):=a+ib$. Then $w_0\in L_{\theta_0}$ and $w_1\in L_{\theta_1}$ for some $\theta_0< \theta_1\in (0,\pi)$ and $\Im \sigma(L_{\theta_0})=\Im \sigma(L_{\theta_1})=\{a\}$. Let $Q:=\{z\in \H: z=re^{i\theta}, r\in [1,2], \theta\in [\theta_0,\theta_1]\}$. The function $Q\ni w\mapsto \Im \sigma(w)$ is harmonic in the interior of $Q$ and continuous on $Q$. Moreover, since $\sigma$ is not constant, $\Im \sigma$ cannot be constant in $Q$. Then $\Im \sigma$ has a maximum or a minimum (or both) different from $a$. Let $w'=r'e^{i\theta'}\in Q$ be such a point, with $\theta'\in (\theta_0,\theta_1)$, $r'\in [1,2]$. Since $\Im\sigma (r'e^{i\theta'})$ does not depend on $r'$, it follows that $\Im \sigma (\frac{3}{2}e^{i\theta'})$ is ether   a maximum or a minimum in the interior of $Q$, a contradiction. Thus $\sigma$ is univalent.

The proof of (2) is similar.
\end{proof}

\begin{proposition}
Let $(\phi_t)$ be a semigroup of holomorphic self-maps of the unit disc $\D$, and let $(\D, \eta,\varphi_t)$ be a Kobayashi hyperbolic semi-model for $(\phi_t)$. If  $(\phi_t)$ is elliptic (resp. parabolic, resp.  hyperbolic with dilation $(e^{-\lambda })$ for some $\lambda>0$), then $(\varphi_t)$ is elliptic (resp. parabolic, resp.  hyperbolic with dilation $(e^{-\lambda })$).
\end{proposition}
\begin{proof}
If $(\phi_t)$ is elliptic, all $\phi_t$'s share a common   fixed point $z_0\in \D$. Then $\eta(z_0)$ is fixed by $\varphi_t$ for all $t\geq 0$, hence $(\varphi_t)$ is elliptic.

Assume that $(\phi_t)$ is hyperbolic with dilation $(e^{\lambda t})$.
Let  $(\H, h, z\mapsto e^{\lambda t}z)$ be a canonical Kobayashi hyperbolic model for $(\phi_t)$. Then there exists a morphism of models $\hat \sigma\colon (\H, h, z\mapsto e^{ct}z)\to (\D, \eta,\varphi_t)$.
The  group of automorphisms $(\v_t)$ cannot be elliptic, since $\sigma$ cannot map a half-line to a point. Since $\sigma\colon \H\to \D$ is surjective, Lemma \ref{inj} shows that $(\v_t)$ is hyperbolic with dilation $(e^{\lambda t})$.

Finally suppose $(\phi_t)$ is parabolic. Then $(\v_t)$ cannot be hyperbolic since we have $c(\phi_t)\geq c(\v_t)$ for all $t\geq 0$ by Lemma \ref{c}.  Let  $(\H, h, z\mapsto z + \varepsilon t)$ be a canonical Kobayashi hyperbolic model for $(\phi_t)$, where  $\varepsilon=\pm 1$. Then there exists a morphism of models $\hat \sigma\colon (\H, h, z\mapsto z + \varepsilon t)\to (\D, \eta,\varphi_t)$.
The group of automorphisms $(\v_t)$ cannot be elliptic, since $\sigma$ cannot map a horizontal line to a point.
\end{proof}

\bibliographystyle{amsplain}

\end{document}